\documentclass[11pt,twoside]{amsart}

\pdfoutput=1
\usepackage[english]{babel}
\usepackage{csquotes}
\usepackage{amsfonts}
\usepackage{hyperref}
\usepackage{etoolbox}
\usepackage{amsmath}
\usepackage{amssymb}
\usepackage{amsthm}
\usepackage{stmaryrd}
\usepackage{dsfont}
\usepackage{pifont}
\usepackage{caption}
\usepackage{subcaption}
\usepackage{mathrsfs}
\usepackage{graphicx}
\usepackage{enumerate}
\usepackage[all]{xy}
\usepackage{enumitem}
\usepackage{geometry}
\usepackage{amsfonts}
\usepackage{fancyhdr}
\usepackage{calc}
\usepackage{cancel}
\usepackage{accents}
\usepackage{abraces}
\usepackage{amsaddr}
\usepackage{orcidlink}
\usepackage{tikz,tikz-3dplot}
\usetikzlibrary{decorations.markings}
\usetikzlibrary{shapes,arrows,plotmarks}
\usepackage[backend=biber,style=alphabetic,sorting=nty]{biblatex}

\definecolor{cof}{RGB}{219,144,71}
\definecolor{pur}{RGB}{186,146,162}
\definecolor{greeo}{RGB}{91,173,69}
\definecolor{greet}{RGB}{52,111,72}

\tdplotsetmaincoords{70}{165}

\newcommand{\changefont}{%
    \fontsize{8}{8}\selectfont
}
\newcommand{\inabove}{\rotatebox[origin=c]{90}{$\in$}}

\title[Fundamental polytopes for Weyl groups on maximal tori of compact Lie groups]{Fundamental polytope for the Weyl group acting on a maximal torus of a compact Lie group}

\author{Arthur Garnier \orcidlink{0000-0003-4069-3203}}
\address{LAMFA, Universit\'e de Picardie Jules Verne, CNRS UMR 7352, \\33, rue Saint-Leu, 80000, Amiens, France.}
\email{arthur.garnier@math.cnrs.fr}

\theoremstyle{plain}
\newtheorem{prop}{Proposition}[subsection]
\newtheorem{prop-def}[prop]{Proposition-Definition}

\newtheorem{lem}[prop]{Lemma}
\newtheorem{theo}[prop]{Theorem}
\newtheorem{cor}[prop]{Corollary}
\newtheorem{rem}[prop]{Remark}
\newtheorem{definition}[prop]{Definition}
\newtheorem{exemple}[prop]{Example}

\newtheorem*{prop*}{Proposition}
\newtheorem*{prop-def*}{Proposition-Definition}
\newtheorem*{propri*}{Property}
\newtheorem*{lem*}{Lemma}
\newtheorem*{theo*}{Theorem}
\newtheorem*{cor*}{Corollary}
\newtheorem*{rem*}{Remark}
\newtheorem*{definition*}{Definition}
\newtheorem*{exemple*}{Example}
\newtheorem*{notation*}{Notation}

\newcommand{\lra}{\longrightarrow}

\newcommand{\ra}{\rightarrow}
\newcommand{\sdp}{\times\kern-.2em\vrule height1.1ex depth-.05ex}
\newcommand{\epi}{\lra \kern-.8em\ra}
\newcommand{\C}{{\mathbb C}}

\newcommand{\R}{{\mathbb R}}

\newcommand{\Z}{{\mathbb Z}}

\newcommand{\codim}{\mathrm{codim}\,}

\newcommand{\ima}{\mathrm{im}\,}

\newcommand{\ho}{\mathrm{Hom}\,}
\newcommand{\pr}{\mathrm{pr}\,}

\newcommand{\longto}{\longrightarrow}

\newcommand{\Sph}{\mathbb{S}}

\newcommand{\rk}{\mathrm{rk}\,}

\newcommand{\lie}{\mathrm{Lie}\,}
\newcommand{\Sym}{\mathfrak{S}}

\newcommand{\Lie}[1]{\mathfrak{#1}}

\DeclareMathOperator\conv{conv}
\DeclareMathOperator\vertices{vert}

 \normalbaselines
\makeatletter
\newlength\@SizeOfCirc%
\newcommand{\CircleArrowRight}[1]{%
    \setlength{\@SizeOfCirc}{\maxof{\widthof{#1}}{\heightof{#1}}}%
    \tikz [x=1.0ex,y=1.0ex,line width=.12ex]%
        \draw [->,anchor=center]%
            node (0,0) {#1}%
            (0,0.8\@SizeOfCirc) arc (85:-240:0.8\@SizeOfCirc);%
}%
\newcommand{\CircleArrowLeft}[1]{%
    \setlength{\@SizeOfCirc}{\maxof{\widthof{#1}}{\heightof{#1}}}%
    \tikz [x=1.0ex,y=1.0ex,line width=.12ex]%
        \draw [<-,anchor=center]%
            node (0,0) {#1}%
            (0,0.8\@SizeOfCirc) arc (85:-240:0.8\@SizeOfCirc);%
}%
\makeatother


\date{\today}

\begin{document}

\begin{abstract}
We provide a fundamental domain for the action of the finite Weyl group on a maximal torus of a compact Lie group of the corresponding type. The general situation is reduced to the adjoint case and, from the perspective of root data, this problem can be rephrased by asking for a fundamental polytope for the action of the extended affine Weyl group on the (dual) toral subalgebra. We solve the problem in this second form.

Using the theory of minuscule weights, we obtain a description of this fundamental polytope as a convex hull of explicit vertices, and as an intersection of closed half-spaces. The latter description was first obtained by Komrakov and Premet in 1984 but, as the present work is independent of that of Komrakov--Premet, we give a new self-contained proof of it. We also derive some consequences on the structure of automorphism groups of extended Dynkin diagrams.
\end{abstract}

\maketitle

\tableofcontents

\section*{Introduction and motivations}

One of the problems the author studied in his PhD thesis \cite{manuscript} was to exhibit explicit triangulations of maximal tori of a compact Lie group, that are equivariant with respect to the natural action of the corresponding Weyl group.

A first attempt toward this goal was to look for a polytopal fundamental domain for this action\footnote{As explained in \cite[Remark 4.2.5]{manuscript}, the resulting fundamental domain fails to always yield a Weyl-equivariant triangulation in the adjoint case, because some cells are stabilized by elements of the Weyl group that do not restrict to the identity on these cells. The smallest Lie group in which this issue occurs is the projective unitary symplectic group $PSp(3)$ of type $C_3$. The general problem was finally solved by considering abstract barycentric subdivisions, see \cite[Theorem 4.2.3]{manuscript}.}; this is the aim of the present work. Pulling the situation back along the exponential map translates the problem into finding a fundamental polytope for the action of the Weyl group, extended by the kernel of the exponential map (a lattice standing between coroots and coweights). As explained in the final subsection \ref{subsec:general}, this problem can be further reduced to the case of the adjoint Lie group, i.e. the corresponding lattice is the coweight lattice and the extended group is the usual \textit{extended affine Weyl group}. Observe that in the other extremal case (where the Lie group is simply-connected), the extended group is the classical \textit{affine Weyl group} and the fundamental alcove is the desired fundamental polytope.

As is well-known, the extended affine Weyl group is not (naturally) Coxeter, which complicates the combinatorics of its action on the toral Lie subalgebra. However, we shall see that we can uniformly describe a fundamental polytope for this action, as an intersection of half-spaces as well as the convex hull of explicit vertices. As the author realized after writing the proofs down, the first description is originally due to B. P. Komrakov and A. A. Premet in \cite{komrakov-premet}. The main result may be summarized as follows:
\begin{theo*}[{Corollaries \ref{funddomforextendedaffineweylgroup}, \ref{explicit_action}, \ref{funddomforKadjointcase} and Proposition \ref{verticesofFP}}]
Let $\Phi\subset V$ be an irreducible root system, with simple system $\Pi$, highest root $\alpha_0=\sum_{\alpha\in\Pi}n_\alpha\alpha$ and finite Weyl group $W$. We denote by $Q:=\Z\Phi\subset V$ the root lattice and by $P^\vee:=Q^\wedge\subset V^*$ the coweight lattice, which is the dual lattice of $Q$.

The following convex polytope
\[F:=\{\lambda\in V^*~;~\left<\lambda,\alpha_0\right>\le1,~\forall \alpha\in\Pi,~\left<\lambda,\alpha\right>\ge0\text{ and }n_\alpha=1~\Rightarrow~\left<\lambda,\alpha_0+\alpha\right>\le1\}\]
is a (closed) fundamental domain for the extended affine Weyl group $\widehat{W}_{\rm a}=P^\vee\rtimes W$ acting on $V^*$. It is homeomorphic to its image under the canonical projection $V^*\twoheadrightarrow V^*/P^\vee$, yielding a polytopal fundamental domain for the action on $W$ on $V^*/P^\vee$, the latter $W$-space being naturally isomorphic to a maximal torus of the adjoint Lie group of type $\Phi$.

Moreover, letting $\Pi_m:=\{\alpha\in\Pi~;~n_\alpha=1\}$ be the set of simple roots $\alpha$ whose associated fundamental coweight $\varpi_\alpha^\vee$ is minuscule, the set of vertices of the polytope $F$ is given by
\[\vertices(F)=\{0\}\cup\left\{\frac{\varpi_\alpha^\vee}{n_\alpha}\right\}_{\alpha\in\Pi\setminus\Pi_m}\cup\left\{\frac{1}{|\Pi'|+1}\sum_{\beta\in\Pi'}\varpi_\beta^\vee\right\}_{\emptyset\subseteq\Pi'\subseteq\Pi_m}.\]
In other words, the vertices of $F$ are the non-minuscule vertices of the fundamental alcove, together with all isobarycenters of minuscule coweights along with $0$, having a non-zero coefficient with respect to the origin.

Finally, the partial action of $\widehat{W_{\rm a}}$ on $F$ can be described in terms of automorphisms of the extended Dynkin diagram of $\Phi$; see the Table \ref{extendeddynkindiagrams}.
\end{theo*}

The layout of the present work is as follows: first, after some reminders and reformulation of the problem in terms of root data, we review in \S \ref{subsec:simp_connected} the classical simply-connected case, where the usual fundamental alcove is the requested polytope.

Next, the \S \ref{subsec::adjoint} is devoted to the proof of the above theorem, following more or less the same lines as in \cite{komrakov-premet} for the half-spaces description. The overall strategy is to use an explicit affine description of the non-trivial elements of the stabilizer $\Omega$ in $\widehat{W}_{\rm a}$ of the fundamental alcove. In fact, we prove that $F$ is a fundamental polytope for $\Omega$ acting on this alcove. As a side consequence, we obtain the Corollary \ref{pi(Omega)=Aut(Dyn0)}, stating that the automorphism group of the extended Dynkin diagram splits as the semi-direct product of $\Omega$ by the automorphism group of the (finite) Dynkin diagram.

Finally, we explain in \S \ref{subsec:general} how to generalize this approach to arbitrary compact Lie groups of type $\Phi$, using the correspondence between subgroups of $\Omega$ and lattices standing between roots and weights.

\addtocontents{toc}{\protect\setcounter{tocdepth}{1}}
\subsection*{Acknowledgments}
I would like to warmly thank Prof. Premet for having kindly sent to me a scanned copy of his 1984 article, for useful discussions that helped me improving the present notes, as well as for allowing me to put them on arXiv.
\addtocontents{toc}{\protect\setcounter{tocdepth}{2}}

\section{Prerequisites and notation}

\subsection{Root data}
\hfill

We start by briefly recalling what a \emph{root datum} is and how one can associate a root datum to any connected reductive complex algebraic group (and more specifically to any semisimple compact Lie group). Standard references for what follows are \cite{malle-testerman} and \cite{kirillovandjunior}.

\begin{definition}\label{rootdatum}\emph{(\cite[Definition 9.10]{malle-testerman})}
A \emph{root datum} is a quadruple $(X,\Phi,Y,\Phi^\vee)$ where
\begin{enumerate}[label=(RD\arabic*)]
\item the elements $X$ and $Y$ are free abelian groups of finite rank, together with a perfect pairing $\left<\cdot,\cdot\right> : Y\times X \longto \Z$,
\item the subsets $\Phi\subset X$ and $\Phi^\vee\subset Y$ are (abstract) reduced root systems in $\Z\Phi\otimes_\Z\R$ and $\Z\Phi^\vee\otimes_\Z\R$, respectively,
\item there is a bijection $\Phi\longto\Phi^\vee$ (denoted by $\alpha\longmapsto \alpha^\vee$) such that $\left<\alpha^\vee,\alpha\right>=2$ for every $\alpha\in\Phi$,
\item the reflections $s_\alpha$ of the root system $\Phi$ and $s_{\alpha^\vee}$ of $\Phi^\vee$ are respectively given by
\[\forall x\in X,~s_\alpha(x):=x-\left<\alpha^\vee,x\right>\alpha\]
and
\[\forall y\in Y,~s_{\alpha^\vee}(y):=y-\left<y,\alpha\right>\alpha^\vee.\]
\end{enumerate}
The Weyl group $W$ of the root system $\Phi$ (which is isomorphic to the Weyl group of $\Phi^\vee$ via the map $s_\alpha\longmapsto s_{\alpha^\vee}$) is called the \emph{Weyl group of the root datum}. Moreover, we say that the root datum $(X,\Phi,Y,\Phi^\vee)$ is \emph{irreducible} if the root system $\Phi$ is.
\end{definition}

From this, one can easily define a morphism of root data and the corresponding category of root data.

For a root datum $(X,\Phi,Y,\Phi^\vee)$, we denote by $V:=\Z\Phi\otimes_\Z\R$ the ambient space and $V^*:=\Z\Phi^\vee\otimes_\Z\R$ (the notation is consistent since $\Z\Phi^\vee\otimes_\Z\R$ may be identified with the dual of $\Z\Phi\otimes_\Z\R$, via the pairing $\left<\cdot,\cdot\right>$). As usual, we denote by $\Phi^+\subset\Phi$ the set of positive roots (with respect to some linear order on $V$) and by $\Pi\subset\Phi^+$ the corresponding set of simple roots. Recall that there are elements $\varpi_\alpha\in V$ indexed by $\alpha\in\Pi$ such that
\[\forall \beta\in\Pi,~(\beta^\vee,\varpi_\alpha)=\delta_{\alpha,\beta}=\left\{\begin{array}{cc} 1 & \text{if}~\alpha=\beta \\ 0 & \text{otherwise}\end{array}\right.\]
These elements form a basis of $V$ and are called the \emph{fundamental weights} of $\Phi$. Dually, we can define the \emph{fundamental coweights} $\varpi_\alpha^\vee\in V^*$ of $\Phi$ by the property
\[\forall\beta\in\Pi,~(\varpi_\alpha^\vee,\beta)=\delta_{\alpha,\beta}.\]
We also consider respectively
\[Q:=\Z\Phi=\bigoplus_{\alpha\in\Pi}\Z\alpha\subset V~~\text{and}~~Q^\vee:=\Z\Phi^\vee=\bigoplus_{\alpha\in\Pi}\Z\alpha^\vee\subset V^*\]
the \emph{root lattice} and the \emph{coroot lattice} of $\Phi$. Further, we have the respective \emph{weight lattice} and \emph{coweight lattice}:
\[P:=(Q^\vee)^\wedge=\{x\in V~;~\forall\alpha\in\Phi,~\left<\alpha^\vee,x\right>\in\Z\}=\bigoplus_{\alpha\in\Pi}\Z \varpi_\alpha\subset V~~\text{and}~~P^\vee:=\bigoplus_{\alpha\in\Pi}\Z\varpi_\alpha^\vee\subset V^*.\]
Thus, the abelian group $X$ is a $W$-lattice between $Q$ and $P$:
\[Q\subseteq X\subseteq P.\]
If we enumerate the simple roots $\Pi=\{\alpha_1,\dotsc,\alpha_n\}$ (with $n=\rk(X)=\dim(V)$) and if $C:=\left(\left<\alpha_i^\vee,\alpha_j\right>\right)_{1\le i,j\le n}$ is the Cartan matrix of $\Phi$, then we have
\[\det(C)=[P:Q]=[P^\vee:Q^\vee].\]

\begin{rem}\label{justlikeaWlattice}
Note that to give a root datum $(X,\Phi,Y,\Phi^\vee)$ is the same as to give a Euclidean root system $\Phi$ together with a $W$-lattice $X$ such that $Q\subseteq X\subseteq P$. Indeed, given a root datum $(X,\Phi,Y,\Phi^\vee)$, the following bilinear form 
\[(x,y):=\sum_{\alpha\in\Phi}\left<\alpha^\vee,x\right>\left<\alpha^\vee,y\right>\]
certainly defines a $W$-invariant inner product on the ambient space $V$, makes $\Phi$ into a Euclidean root system and $X$ is clearly a lattice in $V$. On the other hand, given a Euclidean root system and a $W$-lattice $Q\subseteq \Lambda\subseteq P$, the inner product yields a perfect pairing $\Lambda^\wedge\times\Lambda \longto \Z$ and then $(\Lambda,\Phi,\Lambda^\vee,\Phi^\vee)$ is indeed a root datum.
\end{rem}

The crucial interest of root data relies in the following theorem of Chevalley:
\begin{theo}\label{chevalley}\emph{(Chevalley classification theorem, \cite[\S 9.2]{malle-testerman})}
Let $G$ be a connected reductive algebraic group over an algebraically closed field $k$ and $T$ be a maximal torus of $G$. Let $\Phi$ be the root system associated to the pair $(G,T)$ and denote by $\Phi^\vee:=\{\alpha^\vee,~\alpha\in\Phi\}$ its dual root system. Let moreover $X(T):=\ho(T,\mathbb{G}_m)$ and $Y(T):=\ho(\mathbb{G}_m,T)$ be the \emph{character lattice} and the \emph{cocharacter lattice} of $T$, respectively. Then, $(X(T),\Phi,Y(T),\Phi^\vee)$ is a root datum.

Two semisimple linear algebraic groups are isomorphic if and only if their root data are isomorphic. For each root datum, there is a semisimple algebraic group which realizes it. Finally, the group is simple if and only if the associated root datum is irreducible.
\end{theo}

This can also be adapted to the study of Lie groups.   As usual, a small gothic letter denotes the Lie algebra of the algebraic (or Lie) group denoted by the same letter, written in capital standard font.

We let $G$ be a semisimple connected algebraic group over $\C$, $B$ be a Borel subgroup of $G$, $K$ be a compact real form of $G$ (i.e. $K$ is a semisimple compact Lie group such that $\Lie{g}=\Lie{k}\otimes_\R \C$). Then, the subgroup $T:=K\cap B$ is a maximal torus of $K$ and the complexified Lie algebra $\Lie{h}:=\Lie{t}\otimes \C$ is a Cartan subalgebra of $\Lie{g}$. 

If $\Phi$ denotes the root system of $(\Lie{k},\Lie{t})$ (which is just the real form of the root system of $(\lie{g},\Lie{h})$) and $\Phi:=\{\alpha^\vee,~\alpha\in\Phi\}$, then $\Phi\subset i\Lie{t}^*$ and we may take
\[X(T):=\{d\lambda : \Lie{t}\longto i\R,~\lambda\in\ho(T,\Sph^1)\}\subset i\Lie{t}^*\]
the \emph{character lattice} of $T$ and similarly, $Y(T)=X(T)^\wedge:=\{x\in i\Lie{t}~;~\forall \lambda\in X(T),~\lambda(x)\in\Z\}\subset i\Lie{t}$ is the \emph{cocharacter lattice} of $T$ and the pairing is of course given by 
\[X(T)\times Y(T) \ni (\lambda,x)\longmapsto \lambda(x)\in\Z.\]
Then $(X(T),\Phi,Y(T),\Phi^\vee)$ is a root datum. Note that we have isomorphisms $W\simeq N_K(T)/T\simeq N_G(T^\C)/T^\C$ and the Killing form $(\cdot,\cdot)$ on $\Lie{g}$ restricts to the Killing form on $\Lie{k}$ and gives a $W$-invariant inner product on $V=i\Lie{t}^*$.

Notice finally that $W$ acts naturally on $T$ by conjugation by a representative in the normalizer $N_K(T)$. This is well-defined since $T$ is abelian. On the other hand, $W$ acts on $V$, $V^*$ and on the lattices $X(T)$ and $Y(T)$.

We have the following important result:
\begin{lem}\label{ker(exp)}\emph{(\cite[Lemma 1]{kirillovandjunior})}
If $x\longmapsto e^x$ denotes the usual (Lie theoretic) exponential map $\Lie{t}\longto T$, then the \emph{normalized exponential map}
\[\begin{array}{ccccc}
\exp & : & i\Lie{t} & \longto & T \\ & & x & \mapsto & e^{2i\pi x}\end{array}\]
is surjective and descends to an isomorphism of Lie groups
\[V^*/Y(T) \stackrel{\tiny{\sim}}\longto T.\]
Furthermore, this isomorphism is $W$-equivariant.
\end{lem}

We have the following important isomorphisms
\begin{equation}\label{pi1andZ}
P/X(T)\simeq \pi_1(K)~~\text{and}~~X(T)/Q\simeq Z(K).
\end{equation}
In particular, the product $|\pi_1(K)|\times|Z(K)|=[P:Q]=\det(C)$ is constant on the isogeny class of $K$.

\begin{theo}\label{X(T)charactK}\emph{(\cite[Theorem 7]{kirillovandjunior})}
The group $K$ is determined (up to isomorphism) by its Lie algebra $\Lie{k}$, the maximal toral subalgebra $\Lie{t}$ of $\Lie{k}$ and by the lattice $X(T)\subset i\Lie{t}^*$, which can be any lattice $L$ such that $Q\subseteq L\subseteq P$.

The group corresponding to $L=P$ is the \emph{simply-connected group} and the one associated to $L=Q$ is the \emph{adjoint group}.
\end{theo}

The previous two results show that the initial problem of finding a fundamental domain for $W$ on the torus $T$ may be reformulated as follows: given a root datum $(X,\Phi,Y,\Phi^\vee)$, with Weyl group $W$ and ambient space $V:=\Z\Phi\otimes_\Z\R$, find a fundamental domain for $W$ acting on the torus $V^*/Y$.

In this context and, in view of the isomorphisms (\ref{pi1andZ}), the group $P/X$ is called the \emph{fundamental group} of the root datum. Notice that this yields a combinatorial way of defining the fundamental group of any connected reductive algebraic group.

Therefore, we shall fix the following notations:
\begin{notation*}
Throughout the paper we fix, once and for all, an irreducible root datum $(X,\Phi,Y,\Phi^\vee)$, with ambient space $V=\Z\Phi\otimes \R$, simple roots $\Pi\subset\Phi^+$, Weyl group $W=\left<s_\alpha,~\alpha\in\Pi\right>$, fundamental (co)weights $(\varpi_\alpha)_{\alpha\in\Pi}$ and $(\varpi^\vee_\alpha)_{\alpha\in\Pi}$, (co)root lattices $Q$ and $Q^\vee$ and (co)weight lattices $P$ and $P^\vee$, just as in the beginning of this section. Moreover, we consider the natural projection (which is $W$-equivariant):
\[\pr : V^*\twoheadrightarrow V^*/Y\]
\end{notation*}

\subsection{Reduction to fundamental domains of (extended) affine Weyl groups}
\hfill

For a vector $v$ is some vector space, we denote by $\mathrm{t}_v$ the affine endomorphism of the vector space defined by the translation by the vector $v$. Since $X$ is a $W$-lattice that stands between $Q$ and $P$, we may consider the \emph{intermediate affine Weyl group}
\[W_X:=\mathrm{t}(X)\rtimes W.\]
It is a subgroup of the group $\mathrm{Aff}(V)$ of affine transformations of the space $V$. In the same way, we may consider the group $W_Y:=\mathrm{t}(Y)\rtimes W$ of $\mathrm{Aff}(V^*)$ and we have a canonical isomorphism $W_X\simeq W_Y$, induced by the pairing $\left<\cdot,\cdot\right>$.

In the sequel, by a \emph{fundamental domain} for the action of a	 (discrete) group $G$ on a topological space $Z$, we mean a closed connected subspace $\mathcal{F}\subseteq X$ such that:
\begin{enumerate}[label=$\ast$]
\item For $1\ne g\in G$ the subset $\mathcal{F}\cap g\mathcal{F}$ has empty interior.
\item The translates of $\mathcal{F}$ cover the whole space: $Z=\bigcup_{g\in G}g\mathcal{F}$.
\end{enumerate}

\begin{lem}\label{fromlatticetotorus}
If $F$ is a fundamental domain for the action of $W_Y$ on $V^*$, then the subset $\mathcal{F}:=\pr(F)$ of $V^*/Y$ is a fundamental domain for the induced action of the Weyl group $W$.

Moreover, if $F\cap(Y\setminus\{0\})+F)=\emptyset$, then the restricted map
\[\pr : F\longto\mathcal{F}\]
is a homeomorphism.
\end{lem}
\begin{proof}
The fact that the $W$-translates of $\mathcal{F}$ do cover is obvious. Now let $1\ne w\in W$ and let us prove that $\mathcal{F}\cap w\mathcal{F}$ has no interior. Since the map $\pr$ is onto, it suffices to show that $\pr^{-1}(\mathcal{F}\cap w\mathcal{F})$ has empty interior in $V^*$. We compute
\[\pr^{-1}(\mathcal{F}\cap w\mathcal{F})=\pr^{-1}(\mathcal{F})\cap\pr^{-1}(w\mathcal{F})=(F+Y)\cap(wF+Y)\]
\[=\left(\bigcup_{x\in Y}\mathrm{t}_x(F)\right)\cap\left(\bigcup_{y\in Y}\mathrm{t}_y(wF)\right)=\bigcup_{x,y\in Y}\mathrm{t}_x(F\cap\mathrm{t}_{y-x}(wF)).\]
But since $W$ is finite, we have $W\cap\mathrm{t}(Y)=1$ and because $w\ne 1$, we have $w\notin \mathrm{t}(Y)$. As $F$ is a fundamental domain for $W_Y$, given each $y\in Y$, the closed set $F\cap \mathrm{t}_ywF$ has empty interior in $V^*$ and well as $\mathrm{t}_y(F\cap\mathrm{t}_{y-x}(wF))$. Since $Y$ is countable, the conclusion follows from Baire's theorem. The second statement is a straightforward verification.
\end{proof}

\begin{rem}
Of course, the previous result is \emph{auto-dual}, i.e. it still holds if we replace $V$ by $V^*$ and $Y$ by $X$.
\end{rem}

\subsection{The affine Weyl group}
\hfill

Standard references for what follows are \cite{bourbaki456}, \cite{humphreys-reflectiongroups}.

The affine Weyl group is an infinite extension of the Weyl group $W$ that contains more information about the root system $\Phi$ than $W$ but still has some of the good properties of $W$; namely it is still a Coxeter group. One can use the affine Weyl group to determine the subsystems of $\Phi$ and to compute the order of $W$. What follows is quite classical and may be found for example in \cite{kane}, \cite{humphreys-reflectiongroups} and \cite{bourbaki456}.

First, recall that the root system $\Phi$ admits a \emph{highest root}. That is, there exists a unique root $\alpha_0\in\Phi$ such that $\alpha\le\alpha_0$ for every $\alpha\in\Phi$. For this, see \cite[\S 10.4, Lemma A]{humlie} or \cite[\S 11.2]{kane}. Notice that some authors use the notation $\widetilde{\alpha}$ for the highest root, but the we rather choose $\alpha_0$ for notational purposes.

Given $\alpha\in\Phi$ and $k\in\Z$, define the following hyperplane of $V^*$:
\[H_{\alpha,k}:=\{\lambda\in V^*~;~\left<\lambda,\alpha\right>=k\}.\]
We also consider the orthogonal reflection $s_{\alpha,k}$ with respect to this hyperplane:
\[s_{\alpha,k}:=s_{\alpha}+k\mathrm{t}_{\alpha^\vee} : \lambda\mapsto \lambda-(\left<\lambda,\alpha\right>-k)\alpha^\vee,\]
with $\mathrm{t}_{\alpha^\vee}(\lambda):=\lambda+\alpha^\vee$.

\begin{definition}\label{defAWG}\emph{(\cite[\S 11.1]{kane})}
Define the \emph{affine root system} by $\Phi_{\mathrm{a}}:=\Phi\times\Z$, the \emph{positive affine roots} by $\Phi_{\mathrm{a}}^+:=\{(\alpha,n)\in\Phi_{\mathrm{a}}~;~n>0~\text{or}~n=0,~\alpha\in\Phi^+\}$ and the \emph{negative affine roots} by $\Phi_{\mathrm{a}}^-:=\Phi_{\mathrm{a}}\setminus\Phi_{\mathrm{a}}^+$.

If $\Pi=\{\alpha_1,\dotsc,\alpha_n\}$ and if $\alpha_0\in\Phi$ is the highest root of $\Phi$, then we consider the (affine) reflections
\[\forall 1\le i \le n,~s_i:=s_{\alpha_i,0}~~\text{and}~~s_0:=s_{\alpha_0,1}.\]
Then, the \emph{affine Weyl group} is the subgroup of affine transformations $\mathrm{Aff}(V^*)$ generated by these reflections:
\[W_{\mathrm{a}}:=\left<s_0,s_1\dotsc,s_n\right>.\]
\end{definition}

\begin{theo}\label{affineweylissemidirectandcoxeter}\emph{(\cite[\S 11.3]{kane})}
We have the following isomorphism
\[W_{\mathrm{a}}\simeq Q^\vee\rtimes W.\]
Furthermore, if $m_{i,j}$ denotes the order of $s_is_j$ in $W_{\mathrm{a}}$ (for $0\le i,j \le n$), then we have a presentation
\[W_{\mathrm{a}}=\left<s_0,s_1,\dotsc,s_n~|~(s_is_j)^{m_{i,j}}=1,~\forall i,j\right>.\]
In particular, $W_{\mathrm{a}}$ is a Coxeter group.
\end{theo}

Similarly to the finite case, one defines the \emph{length} of an element of $W_{\mathrm{a}}$ as follows:
\[\forall w\in W_{\mathrm{a}},~\ell(w):=\left|\Phi_{\mathrm{a}}^-\cap w\Phi_{\mathrm{a}}^+\right|.\]

In this context, we have the important notion of an \emph{alcove}.
\begin{definition}\label{defalcove}\emph{(\cite[\S 11.5]{kane})}
An \emph{alcove} is the closure of a connected component of the complementary of the reflecting hyperplanes, i.e. a connected component of the set 
\[V^*\setminus\bigcup_{\substack{\alpha\in\Phi^+ \\ k\in\Z}}H_{\alpha,k}\]
Among these, we consider the \emph{(closed) fundamental alcove}:
\begin{align*}
\mathcal{A}_0:&=\{\lambda\in V^*~;~\forall \alpha\in\Phi^+,~0\le\left<\lambda,\alpha\right>\le1\} \\
&=\{\lambda\in V^*~;~\forall \alpha\in\Pi,~\left<\lambda,\alpha\right>\ge0~\text{and}~\left<\lambda,\alpha_0\right>\le1\}
\end{align*}
A subset of the boundary of $\mathcal{A}_0$ of the form $\partial{\mathcal{A}_0}\cap H_{\alpha,0}$ or $\partial{\mathcal{A}_0}\cap H_{\alpha,1}$, for some $\alpha\in\Phi^+$, is called a \emph{wall} of $\mathcal{A}_0$.
\end{definition}

\begin{rem}\label{A0ispolytope}
If we decompose the highest root $\alpha_0\in\Phi^+$ in the basis $\Pi$ as
\[\alpha_0=\sum_{\alpha\in\Pi} n_\alpha \alpha,\]
then, it is shown in \cite[V, \S 2.2, Corollaire]{bourbaki456} that ${\mathcal{A}_0}$ is the convex polytope given by
\[{\mathcal{A}_0}=\conv\left(\{0\}\cup\left\{\frac{\varpi_\alpha^\vee}{n_\alpha}\right\}_{\alpha\in\Pi}\right).\]
\end{rem}

The importance of the fundamental alcove relies on the following result:
\begin{theo}\label{fundalcovefunddom}\emph{(\cite[\S 4.5 and \S 4.8]{humphreys-reflectiongroups})}
The closed fundamental alcove ${\mathcal{A}_0}$ is a fundamental domain for $W_{\mathrm{a}}$ on $V^*$. Moreover, the group $W_\mathrm{a}$ acts simply transitively on open alcoves.
\end{theo}

Note the analogy with the notion of a chamber. Recall that a \emph{chamber} is a connected component of $V^*\setminus\bigcup_{\alpha\in\Phi^+}H_{\alpha,0}$. A fundamental domain for the action of $W$ on $V^*$ is given by the \emph{(closed) fundamental Weyl chamber} ${\mathcal{C}_0}$ defined by
\[\mathcal{C}_0=\{\lambda\in V^*~;~\forall\alpha\in\Pi,~\left<\lambda,\alpha\right>\ge0\}=\left\{\sum_{\alpha\in\Pi}\lambda_\alpha\varpi_\alpha~;~\lambda_\alpha\ge0\right\}.\]
This can be found in \cite[\S 5.2]{kane}.

\section{Fundamental domain for the action of $W$ on $V^*/Y$}

\subsection{The simply-connected case $Y=Q^\vee$}\label{subsec:simp_connected}
\hfill

The main result of this section is the following one:
\begin{prop}\label{WfunddomforTsimplyconnectedcase}
Assume that the irreducible root datum $(X,\Phi,Y,\Phi^\vee)$ is simply-connected (that is, $X=P$ is the weight lattice). Then the set $\pr(\mathcal{A}_0)$ is a fundamental domain for the action of $W$ on $V^*/Y$. Moreover the map restricted projection yields a homeomorphism $\pr(\mathcal{A}_0)\simeq\mathcal{A}_0$.
\end{prop}

Only the second part of the last statement deserves a proof. First, we need a small technical lemma:
\begin{lem}\label{diffovA0elementsarelinkedbyW}
The following holds
\[\forall \lambda,\lambda'\in{\mathcal{A}_0},~\exists w\in W~;~w(\lambda-\lambda')\in {\mathcal{A}_0}.\]
\end{lem}
\begin{proof}
Let $\mu:=\lambda-\lambda'$. Since, $\lambda,\lambda'\in{\mathcal{A}_0}$, we have $-1\le \left<\mu,\alpha\right>\le 1$ for $\alpha\in\Phi^+$ and this still holds for any root $\alpha\in\Phi$ because $\Phi^-=-\Phi^+$. Thus, by $W$-invariance of the pairing $\left<\cdot,\cdot\right>$, we get 
\[\forall \alpha\in\Phi,~\forall w\in W,~-1\le\left<w(\mu),\alpha\right>\le1\]
and in particular, we have also $\left<w(\mu),\alpha_0\right>\in[-1,1]$.

Now, as ${\mathcal{C}_0}$ is a fundamental domain for the action of $W$ on $V^*$, there is some $w\in W$ such that $w(\mu)\in{\mathcal{C}_0}$. But we have seen that $-1\le\left<w(\mu),\alpha_0\right>\le 1$ and thus $w(\mu)\in\mathcal{A}_0$.
\end{proof}

\begin{rem}
The proof above shows in particular that an element $w\in W$ such that $w(\lambda-\lambda')\in \mathcal{A}_0$ has length $\ell(w)=|\{\alpha\in\Phi^+~;~\left<\lambda,\alpha\right><\left<\lambda',\alpha\right>\}|$. Moreover, since the \emph{longest element} $w_0\in W$ is characterized by $w_0(\Phi^+)=\Phi^-$, we have $-\mathcal{A}_0=w_0\mathcal{A}_0$.
\end{rem}

\emph{Proof of proposition \ref{WfunddomforTsimplyconnectedcase}.} In view of the second part of the Lemma \ref{fromlatticetotorus}, we have to prove that for $\lambda,\lambda'\in\mathcal{A}_0$, we have $\lambda=\lambda'$ as soon as $\lambda-\lambda'\in Q^\vee$. Using the Lemma \ref{diffovA0elementsarelinkedbyW}, we find $w\in W$ such that $w(\lambda-\lambda')\in{\mathcal{A}_0}$. Since $W$ acts on $Q^\vee$, we are left to prove that ${\mathcal{A}_0}\cap Q^\vee=\{0\}$.

Next, recall that the Weyl chamber $\mathcal{C}_0$ is a union of alcoves (see \cite[\S 4.3]{humphreys-reflectiongroups}) and since $W$ acts simply transitively on the set of chambers, for every alcove $\mathcal{A}$ there is a unique $w\in W$ such that $w\mathcal{A}\subset\mathcal{C}_0$. Hence, if $0\in{\mathcal{A}}$, then $0\in{w\mathcal{A}}$ and therefore $w\mathcal{A}=\mathcal{A}_0$ (see \cite[VI, \S 2.2, Proposition 4]{bourbaki456}). We have shown that for any alcove $\mathcal{A}$ containing $0$, there exists $w\in W$ such that $\mathcal{A}=w(\mathcal{A}_0)$.

Finally, let $q\in{\mathcal{A}_0}\cap Q^\vee$. Then $\mathrm{t}_{-q}(\mathcal{A}_0)$ is an alcove containing $0$. The above discussion ensures the existence of $w\in W$ such that $\mathrm{t}_{-q}\mathcal{A}_0=w\mathcal{A}_0$ and since $W_{\mathrm{a}}$ acts simply transitively on the set of alcoves (cf \cite[\S 11.5, Proposition B]{kane}), this implies that $\mathrm{t}_qw=1$ and since $\mathrm{t}(Q^\vee)\cap W=1$, we get $\mathrm{t}_q=1$ so $q=0$.
\qed

We arrive to the main result of this section, which summarizes the discussion above:
\begin{theo}\label{WtriangulationTsimpconncase}
Let $K$ be a simply-connected simple compact Lie group, $T\le K$ a maximal torus, $W:=N_K(T)/T$ its Weyl group, $\mathcal{A}_0\subset i\Lie{t}$ the (dual) fundamental alcove and $\exp : i\Lie{t}\to T$, $x\mapsto e^{2i\pi x}$ the normalized exponential map. Then ${\mathcal{A}_0}$ is an $r$-simplex (with $r=\dim(T)$) and its image $\exp({\mathcal{A}_0})\subset T$ is a fundamental domain for $W$, homeomorphic to ${\mathcal{A}_0}$.
\end{theo}

\subsection{The adjoint case $Y=P^\vee$: the fundamental polytope of Komrakov--Premet}\label{subsec::adjoint}
\hfill

We are looking for a fundamental domain for the action of the \emph{extended affine Weyl group} $\widehat{W_\mathrm{a}}:=\mathrm{t}(P^\vee)\rtimes W$ on $V^*$ in the case where $Y=P^\vee$ is the coweight lattice. Things get trickier in this case, because $\widehat{W_\mathrm{a}}$ is no longer a reflection group in general.

After the study of this question and thanks to \cite{hm}, the author realized that the main result of this section (Corollary \ref{funddomforextendedaffineweylgroup} below) was known since 1984 and is due to Komrakov and Premet \cite{komrakov-premet}. For the sake of self-containment and as the present work is independent of the one of Komrakov--Premet, we give a new full proof. Our method is quite close to the one used in \cite{komrakov-premet} but in addition, we describe the fundamental domain as a convex hull, which is a new result, as far as the author knows.

The group $\widehat{W_\mathrm{a}}$ acts on alcoves (transitively since $W_\mathrm{a}\unlhd\widehat{W_\mathrm{a}}$ does) but not simply-transitively. We introduce the stabilizer
\[\Omega:=\{\phi\in\widehat{W_\mathrm{a}}~;~\phi\cdot\mathcal{A}_0=\mathcal{A}_0\}\]
and we see that we have a decomposition $\widehat{W_\mathrm{a}}\simeq W_\mathrm{a}\rtimes\Omega$. In particular, one has
\[\Omega\simeq \widehat{W_\mathrm{a}}/W_\mathrm{a}\simeq P^\vee/Q^\vee\simeq P/Q.\]
Thus, $\Omega$ is a finite abelian group. The Table \ref{fund_groups} details the fundamental groups of the irreducible root systems.

\begin{center}\label{fundgroups}
\begin{tabular}{|c|c|}
\hline
Type & $\Omega\simeq P/Q$ \\
\hline
\hline
$A_n~(n\ge1)$ & $\Z/(n+1)\Z$ \\
\hline
$B_n~(n\ge2)$ & $\Z/2\Z$ \\
\hline
$C_n~(n\ge3)$ & $\Z/2\Z$ \\
\hline
$D_{2n}~(n\ge2)$ & $\Z/2\Z \oplus \Z/2\Z$ \\
\hline
$D_{2n+1}~(n\ge2)$ & $\Z/4\Z$ \\
\hline
$E_6$ & $\Z/3\Z$ \\
\hline
$E_7$ & $\Z/2\Z$ \\
\hline
$E_8$ & $1$ \\
\hline
$F_4$ & $1$ \\
\hline
$G_2$ & $1$ \\
\hline
\end{tabular}
\captionof{table}{Fundamental groups of irreducible root systems}\label{fund_groups}
\end{center}

The description of $\Omega$ given in \cite[VI, \S 2.3]{bourbaki456} will be useful. First, to simplify the notations, we denote by $r:=\mathrm{rk}(\Phi)=\dim(V)$ the rank of $\Phi$ and we index the simple roots by $I:=\{1,\dotsc,r\}$, that is 
\[\Pi=\{\alpha_1,\dotsc,\alpha_r\}.\]
Given the highest root $\alpha_0=\sum_{i=1}^r n_i\alpha_i$ of $\Phi$, recall that a weight $\varpi_i$ is called \emph{minuscule} if $n_i=1$ and that minuscule weights form a set of representatives of the classes in $P/Q$ (see \cite[Chapter VI, Exercise 24]{bourbaki456}). Dually, we have the same notion and result for \emph{minuscule coweights}. Let 
\[J:=\{i \in I~;~n_i=1\}.\]

\begin{prop-def}\emph{(\cite[VI, \S 2.3, Proposition 6]{bourbaki456})}\label{descriptionOmega}
Let $\alpha_0=\sum_{i\in I} n_i\alpha_i$ be the highest root of $\Phi$ and $w_0\in W$ be the longest element. For $i\in I$, denote by $W_i\le W$ the Weyl group of the subsystem of $\Phi$ generated by $\{\alpha_j~;~j\ne i\}\subset\Pi$. For $i\in J$, let $w^i_0\in W_i$ be the longest element of $W_i$ and $w_i:=w^i_0w_0$. 

Then the element $\mathrm{t}_{\varpi_i^\vee}w_i\in\widehat{W_\mathrm{a}}$ is in $\Omega$ and the map
\[\begin{array}{ccc}
J & \to & \Omega\setminus\{1\} \\ i & \mapsto & \omega_i:=\mathrm{t}_{\varpi_i^\vee}w_i \end{array}\]
is a bijection.
\end{prop-def}

\begin{rem}
Note that, with the preceding notations, $w^i_0$ and $w_0$ have order $2$ and $w_i^{-1}=w_0w^i_0$ sends $\varpi_i^\vee$ to $-{\mathcal{A}_0}$. Indeed, denoting by $\Phi_i:=\Phi\cap\Z\left<\Pi\setminus\{\alpha_i\}\right>$ the subsystem of $\Phi$ generated by the simple roots other than $\alpha_i$, for $j\ne i$ we have $w^i_0\in W(\Phi_i)$ and
\[\left<w_0^i\varpi_i^\vee,\alpha_j\right>=\left<\varpi_i^\vee,w_0^i\alpha_j\right>=0,\]
because $w^i_0\alpha_j\in\Phi_i$. Since $w_0^i\varpi_i^\vee\in P^\vee$ and $W$ stabilizes the set 
\[\{\lambda\in V^*~;~\forall \alpha\in\Phi,~-1\le\left<\lambda,\alpha\right>\le 1\},\]
we have that $\left<w_0^i\varpi_i^\vee,\alpha_i\right>$ is an integer between $-1$ and $1$. Of course, it is not $0$ and if it is $-1$, then we have $w_0^i\alpha_i=-\alpha_i$, so $w_0^i(\Pi)\subset\Phi^-$ and thus $w_0^i=w_0$, a contradiction. Hence, we have $w_0^i\varpi_i^\vee=\varpi_i^\vee$ and thus $w_i^{-1}\varpi_i^\vee=w_0w_0^i\varpi_i^\vee=w_0\varpi_i^\vee\in -{\mathcal{A}_0}$ since $w_0\alpha_i\in\Phi^-$.
\end{rem}

We shall prove the following result:
\begin{prop}[{\cite[Theorem, (1)-(3)]{komrakov-premet}}]\label{funddomforOmega}
The subset of $\mathcal{A}_0$ defined by
\[F_{P^\vee}:=\{\lambda\in\mathcal{A}_0~;~\forall i\in J,~\left<\lambda,\alpha_i+\alpha_0\right>\le1\}\]
is a fundamental domain for the action of $\Omega$ on $\mathcal{A}_0$.

Moreover, for $\lambda\in F_{P^\vee}$, the elements $\omega_i^{-1}\in\Omega$ that send $\lambda$ into $F_{P^\vee}$ are precisely those with $\left<\lambda,\alpha_i+\alpha_0\right>=1$. In particular, if $\lambda$ is an interior point and $1\ne \widehat{w}\in \widehat{W_{\rm a}}$, then $\widehat{w}(\lambda)\notin F_{P^\vee}$.
\end{prop}

This will be done step-by-step. Once and for all, we fix the notation
\[\alpha_0=\sum_{i=1}^rn_i\alpha_i\]
for the highest root of $\Phi$. First, we have to obtain more information about how $\Omega$ acts on $\mathcal{A}_0$ and how its Weyl part acts on roots. The fact that $\Omega$ is formed of affine isomorphisms will considerably constrain these actions. 

\begin{lem}\label{Omegaactsonvertices}
The subgroup $\Omega\le\widehat{W_\mathrm{a}}$ acts on the set $\{0\}\cup\left\{\frac{\varpi_i^\vee}{n_i}\right\}_{i\in I}$ and on the subset $\{0\}\cup\{\varpi_j^\vee\}_{j\in J}$.
\end{lem}
\begin{proof}
Let $\omega\in\Omega$. Since ${\mathcal{A}_0}$ is a convex polytope with vertices $\{0\}\cup\{{\varpi_i^\vee}/{n_i}\}_{i\in I}$ (see \cite[VI, \S 2.2]{bourbaki456}) and since $\omega$ is affine, it sends any vertex of ${\mathcal{A}_0}$ to another one. Hence, $\Omega$ indeed acts on the first set.

Now, we have to see why a minuscule coweight $\varpi_j^\vee$ is sent by $\omega\in\Omega$ to $0$ or a minuscule coweight. We just have seen that, if it is non-zero, then there is some $i$ such that $\omega(\varpi_j^\vee)={\varpi_i^\vee}/{n_i}$. Using the description given in the Proposition \ref{descriptionOmega} for $\Omega$, we can write $\omega=\omega_k=\mathrm{t}_{\varpi_k ^\vee}w_k$ for some $k\in J$. Note that if $i\in\{j,k\}$, then $n_i=1$ and thus $\omega(\varpi_j^\vee)$ is indeed a minuscule coweight, so we may assume that $i\ne j,k$. Let $i'\in I$ and consider the simple reflection $s_{i'} : \lambda \mapsto \lambda-\left<\lambda,\alpha_{i'}\right>\alpha_{i'}^\vee$. We immediately see that $s_{i'}(\varpi_j^\vee)=\varpi_j^\vee-\delta_{i',j}\alpha_j^\vee$ and by induction, we conclude that there are integers $\lambda^k_{i'}\in\Z$ such that
\[w_k(\varpi_j^\vee)=\varpi_j^\vee-\sum_{i'=1}^r\lambda^k_{i'}\alpha_{i'}^\vee.\]
Hence, we compute
\[\frac{\varpi_i^\vee}{n_i}=\omega_k(\varpi_j^\vee)\stackrel{\tiny{\text{def}}}=w_k(\varpi_j^\vee)+\varpi_k^\vee=\varpi_j^\vee+\varpi_k^\vee-\sum_{i'=1}^r\lambda^k_{i'}\alpha_{i'}^\vee.\]
Taking the pairing of this equation against $\alpha_i$ yields
\[\frac{1}{n_i}=-\sum_{i'}\lambda^k_{i'}\left<\alpha_{i'}^\vee,\alpha_i\right>=-\sum_{i'}\lambda^k_{i'}C_{i',i},\]
where $C=(C_{i',i})\in\mathcal{M}_r(\Z)$ is the Cartan matrix. Thus, the sum above is an integer, which implies that $n_i=1$.
\end{proof}

Now, we outline the way an element $w_i\in W$ acts on roots.
\begin{lem}[{\cite[Lemma 1]{komrakov-premet}}]\label{actionwimonroots}
Let $i\in J$ and consider the corresponding element $\omega_i=\mathrm{t}_{\varpi_i^\vee}w_i\in\Omega$. Then the element $w_i^{-1}\in W$ takes $\Pi\setminus\{\alpha_i\}$ into $\Pi$ and $\alpha_i$ to the \emph{lowest root} $-\alpha_0$, that is
\[\left\{\begin{array}{c}
w_i^{-1}(\Pi\setminus\{\alpha_i\})\subset\Pi \\ w_i^{-1}(\alpha_i)=-\alpha_0\end{array}\right.\]
Moreover, if $j\ne i$, then $n_{w_i^{-1}\alpha_j}=n_j$; in other words, the numbers of occurrences of $\alpha_j$ and of $w_i^{-1}(\alpha_j)$ in $\alpha_0$ are the same. In fact, we have
\[\alpha_j=w_i(\alpha_k)~\Leftrightarrow~\frac{\varpi_j^\vee}{n_j}=\omega_i\left(\frac{\varpi_k^\vee}{n_k}\right).\]

Finally, if $j\in I$ is such that $\omega_i(\varpi_j^\vee)=0$, then $w_i^{-1}(\alpha_0)=-\alpha_j$.
\end{lem}
\begin{proof}
First, we prove that $w_i^{-1}(\Pi\setminus\{\alpha_i\})\subset \Pi$. Consider $\alpha_j\in\Pi$ with $j\ne i$. If we have $\varpi_i^\vee=\omega_i(0)=\frac{\varpi_j^\vee}{n_j}$, then $n_j=1$ and $i=j$, which is excluded. Thus, by the Lemma \ref{Omegaactsonvertices}, there is some $k\in I$ such that $\omega_i\left(\frac{\varpi_k^\vee}{n_k}\right)=\frac{\varpi_j^\vee}{n_j}$. We compute
\begin{equation}\label{this=nk/nj}\tag{$\ast$}
\left<\varpi_k^\vee,w_i^{-1}\alpha_j\right>=n_k\left<w_i\left(\frac{\varpi_k^\vee}{n_k}\right),\alpha_j\right>=n_k\left<\frac{\varpi_j^\vee}{n_j}-\varpi_i^\vee,\alpha_j\right>=\frac{n_k}{n_j}.
\end{equation}

Now, let $\ell\ne k$. If $\omega_i\left(\frac{\varpi_\ell^\vee}{n_\ell}\right)=0$, then we have 
\[\left<\varpi_\ell^\vee,w_i^{-1}\alpha_j\right>=-n_\ell\left<\varpi_i^\vee,\alpha_j\right>=0,\]
otherwise, there is some $\ell'\ne j$ such that $\omega_i\left(\frac{\varpi_\ell^\vee}{n_\ell}\right)=\frac{\varpi_{\ell'}^\vee}{n_{\ell'}}$ and we compute similarly
\[\left<\varpi_\ell^\vee,w_i^{-1}\alpha_j\right>=n_\ell\left<w_i\left(\frac{\varpi_\ell^\vee}{n_\ell}\right),\alpha_j\right>=0.\]
In any case, we have
\begin{equation}\label{other=0}\tag{$\ast\ast$}
\forall \ell\ne k,~\left<\varpi_\ell^\vee,w_i^{-1}\alpha_j\right>=0
\end{equation}

Hence, equations (\ref{this=nk/nj}) ans (\ref{other=0}) imply that $\Phi\ni w_i^{-1}\alpha_j=\frac{n_k}{n_j}\alpha_k$, thus $n_j=n_k$ and $w_i^{-1}\alpha_j=\alpha_k\in\Pi$. This proves the first and third statements at once.

We have to see why $w_i\alpha_0=-\alpha_i$. First, since $n_i=1$, the vector $\varpi_i^\vee$ is a vertex of ${\mathcal{A}_0}$, hence the vector $w_i^{-1}\varpi_i^\vee$ is a vertex of $-{\mathcal{A}_0}$, that is, $w_i^{-1}\varpi_i^\vee$ belongs to the set $\{0\}\cup\{-\varpi_j^\vee/n_j\}$. Since $w_i^{-1}\varpi_i^\vee\ne0$, there is some $j$ such that $w_i^{-1}\varpi_i^\vee=-\varpi_j^\vee/n_j$, so we have
\begin{equation}\label{scal=-1}\tag{$\dagger$}
\left<\varpi_i^\vee,w_i\alpha_0\right>=\left<w_i^{-1}\varpi_i^\vee,\alpha_0\right>=-1.
\end{equation}

Now, let $j\ne i$. Here again, we have $\omega_i(0)=\varpi_i^\vee\ne\frac{\varpi_j^\vee}{n_j}$, so there exists $k$ such that $\omega_i\left(\frac{\varpi_k^\vee}{n_k}\right)=\frac{\varpi_j^\vee}{n_j}$. We compute, using the equation (\ref{scal=-1}),
\begin{equation}\label{otherr=0}\tag{$\ddagger$}
\left<\frac{\varpi_j^\vee}{n_j},w_i\alpha_0\right>=\left<w_i\left(\frac{\varpi_k^\vee}{n_k}\right)+\varpi_i^\vee,w_i\alpha_0\right>=\left<\frac{\varpi_k^\vee}{n_k},\alpha_0\right>-1=0.
\end{equation}

Now, we have $w_i\alpha_0\in\Phi$ and the equation (\ref{otherr=0}) ensures that for $j\ne i$, we have $\left<\varpi_j^\vee,w_i\alpha_0\right>=0$, so there is $\mu\in\{\pm1\}$ such that $w_i\alpha_0=\mu\alpha_i$. But if $\mu=1$, then $w_i^{-1}\alpha_i=\alpha_0\in\Phi^+$ and using the first statement, we obtain $w_i^{-1}(\Pi)\subset\Phi^+$, thus $w_i=1$, a contradiction.

To prove the last statement, we compute 
\[\left<\varpi_j^\vee,w_i^{-1}\alpha_0\right>=\left<w_i\varpi_j^\vee,\alpha_0\right>=-\left<\varpi_i^\vee,\alpha_0\right>=-1.\]
If $k\ne j$, then there exists $\ell$ such that $\omega_i\left(\frac{\varpi_k^\vee}{n_k}\right)=\frac{\varpi_\ell^\vee}{n_\ell}$ and we have
\[\left<\frac{\varpi_k^\vee}{n_k},w_i^{-1}\alpha_0\right>=\left<\frac{\varpi_\ell^\vee}{n_\ell}-\varpi_i^\vee,\alpha_0\right>=0.\]
These two equations together show that $w_i^{-1}\alpha_0=-\alpha_j$, as required.
\end{proof}

\begin{lem}\label{actionwionroots}
For $\omega_i=\mathrm{t}_{\varpi_i^\vee}w_i\in\Omega$ and a simple root $\alpha_k\in\Pi$, the root $w_i\alpha_k$ is simple except in the case where $\omega_i(\varpi_k^\vee)=0$, in which case it is the lowest root.
\end{lem}
\begin{proof}
We compute
\[\omega_i^{-1}=w_i^{-1}\mathrm{t}_{-\varpi_i^\vee}=w_i^{-1}\mathrm{t}_{-\varpi_i^\vee}w_iw_i^{-1}=\mathrm{t}_{-w_i^{-1}(\varpi_i^\vee)}w_i^{-1}=\mathrm{t}_{\omega_i^{-1}(0)}w_i^{-1}.\]
Since $\Omega$ acts on $\{0\}\cup\{\varpi_\ell^\vee/n_\ell\}$, there exists $k_0\in J$ such that $\varpi_{k_0}^\vee=\omega_i^{-1}(0)$ and the conclusion follows from the first two statements of the Lemma \ref{actionwimonroots}.
\end{proof}

Recall that the \emph{extended Dynkin diagram} of the root system $\Phi$ is its Dynkin diagram to which we add a new vertex corresponding to $-\alpha_0$, which is linked to the other vertices in the same way as for the classical Dynkin diagram. 

The Lemmae \ref{actionwimonroots} and \ref{actionwionroots} show that the Weyl part of any element of $\Omega$ permutes the vertices of the extended diagram; the corresponding permutations are made explicit in the Table \ref{extendeddynkindiagrams}. Moreover, since the pairing is $W$-invariant, it also preserves the edges of the diagram. Hence we obtain the following result:

\begin{cor}\label{pi(Omega)=Aut(Dyn0)}
The group $\Omega$ acts on the extended Dynkin diagram $\widehat{\mathcal{D}}$, giving rise to a normal inclusion $\Omega\unlhd{\rm Aut}(\widehat{\mathcal{D}})$. Moreover, the factor group ${\rm Aut}(\widehat{\mathcal{D}})/\Omega$ is isomorphic to the automorphism group ${\rm Aut}(\mathcal{D})\le{\rm Aut}(\widehat{\mathcal{D}})$ of the finite Dynkin diagram $\mathcal{D}$ of $\Phi$ (see Table \ref{extendeddynkindiagrams}).

In other words, there is a semi-direct product decomposition
\[{\rm Aut}(\widehat{\mathcal{D}})\simeq\Omega\rtimes{\rm Aut}(\mathcal{D}).\]
\end{cor}

\emph{Proof of Proposition \ref{funddomforOmega}.} By definition of a fundamental domain, we have to prove that the set
\[F:=\{\lambda\in\mathcal{A}_0~;~\forall i\in J,~\left<\lambda,\alpha_0+\alpha_i\right>\le1\}\]
satisfies the following conditions:
\begin{enumerate}[label=\alph*)]
\item the set $F$ is closed in $\mathcal{A}_0$ and connected,
\item for $1\ne \omega\in\Omega$, the set $F\cap \omega F$ has empty interior in $\mathcal{A}_0$,
\item the union of the $\Omega$-translates of $F$ cover $\mathcal{A}_0$, that is $\mathcal{A}_0=\bigcup_{\omega\in\Omega}\omega F.$
\end{enumerate}

To prove $\mathrm{a})$, we just have to notice that $F$ is an intersection of closed half-spaces, so it is a convex polytope.

Let $\omega\in\Omega\setminus\{1\}$. For $\lambda\in F$, writing $\omega=\omega_i=\mathrm{t}_{\varpi_i^\vee}w_i$, we calculate
\[\left<\omega_i(\lambda),\alpha_0+\alpha_i\right>=\left<w_i\lambda+\varpi_i^\vee,\alpha_0+\alpha_i\right>=2+\left<\lambda,w_i^{-1}\alpha_0+w_i^{-1}\alpha_i\right>,\]
but, by the Lemma \ref{actionwimonroots}, we have $w_i^{-1}\alpha_i=-\alpha_0$ and there is some $k\in J$ such that $w_i^{-1}\alpha_0=-\alpha_k\in\Pi$. Since $\lambda\in F$, we have
\[\left<\omega_i(\lambda),\alpha_0+\alpha_i\right>=2+\left<\lambda,w_i^{-1}\alpha_0+w_i^{-1}\alpha_i\right>=2-\underbrace{\left<\lambda,\alpha_0+\alpha_k\right>}_{\le 1}\ge1.\]
This proves that, if $\lambda \in F$, then $\left<\omega_i(\lambda),\alpha_0+\alpha_i\right>\ge1$. Thus, if $\lambda\in F$ and $\omega_i(\lambda)\in F$, then $\left<\omega_i(\lambda),\alpha_0+\alpha_i\right>=1$. Therefore, we have
\[F\cap\omega_i^{-1}F\subseteq \mathcal{A}_0\cap \{\lambda\in V^*~;~\left<\omega_i(\lambda),\alpha_0+\alpha_i\right>=1\},\]
and the later subset has empty interior in $\mathcal{A}_0$, concluding the proof of $\mathrm{b})$.

To establish $\mathrm{c})$, let $\lambda\in\mathcal{A}_0$. We have to find some $\omega\in\Omega$ such that $\lambda\in\omega F$. Choose $i\in J$ such that
\[\left<\lambda,\alpha_i\right>=\max_{j\in J}\left<\lambda,\alpha_j\right>.\]
If $\left<\lambda,\alpha_0+\alpha_i\right>\le1$, then $\lambda\in F$ and $\omega=1$ is the desired element. Otherwise, pick $j$ such that $\omega_i(\varpi_j^\vee)\ne 0$. We have
\[\left<\omega_i^{-1}(\lambda),\alpha_0+\alpha_j\right>=\left<w_i^{-1}(\lambda-\varpi_i^\vee),\alpha_0+\alpha_j\right>=\left<\lambda-\varpi_i^\vee,w_i\alpha_0+w_i\alpha_j\right>=\left<\lambda-\varpi_i^\vee,\alpha_k-\alpha_i\right>,\]
where $w_i\alpha_j=\alpha_k\in\Pi$ for some $k$; which is ensured to exist by the Lemma \ref{actionwionroots}. If $i=k$, then this is zero. Otherwise, we get
\[\left<\omega_i^{-1}(\lambda),\alpha_0+\alpha_j\right>=\left<\lambda-\varpi_i^\vee,\alpha_k-\alpha_i\right>=1+\underbrace{\left<\lambda,\alpha_k-\alpha_i\right>}_{\le0}\le1.\]
Hence, for every $j$ such that $\omega_i(\varpi_j^\vee)\ne0$, we have $\left<\omega_i^{-1}(\lambda),\alpha_0+\alpha_j\right>\le1$. Now, let $j$ such that $\omega_i(\varpi_j^\vee)=0$. Then, by the Lemma \ref{actionwionroots}, we have $w_i\alpha_j=-\alpha_0$ and thus
\[\left<\omega_i^{-1}(\lambda),\alpha_0+\alpha_j\right>=\left<\lambda-\varpi_i^\vee,w_i\alpha_0+w_i\alpha_j\right>=\left<\varpi_i^\vee-\lambda,\alpha_0+\alpha_i\right>=2-\underbrace{\left<\lambda,\alpha_0+\alpha_i\right>}_{>1}<1.\]
Therefore, for every $j\in J$ we have $\left<\omega_i^{-1}(\lambda),\alpha_0+\alpha_j\right>\le1$ and thus $\omega_i^{-1}(\lambda)\in F$.

Finally, take $\lambda\in F_{P^\vee}$. We have already seen that if $\omega_i\in\Omega$ is such that $\omega_i^{-1}(\lambda)\in F_{P^\vee}$, then $\left<\lambda,\alpha_i+\alpha_0\right>=1$. Conversely, assume that the latter equality holds; we aim to prove that $\omega_i^{-1}(\lambda)\in F_{P^\vee}$. For this, take $i\ne j\in J$ and compute
\begin{align*}
\left<\omega_i^{-1}(\lambda),\alpha_j+\alpha_0\right>&=\left<\lambda-\varpi_i^\vee,w_i\alpha_j+w_i\alpha_0\right>=\left<\lambda-\varpi_i^\vee,w_i\alpha_j-\alpha_i\right> \\
&=1-\left<\lambda,\alpha_i\right>+\left<\lambda-\varpi_i^\vee,w_i\alpha_j\right> \\
&=\left<\lambda,\alpha_0\right>+\left<\lambda-\varpi_i^\vee,w_i\alpha_j\right>.
\end{align*}
Now, two possibilities arise: whether $\omega_i(\varpi_j^\vee)=0$, in which case $w_i\alpha_j=-\alpha_0$, or $w_i\alpha_j=:\alpha_k\in\Pi\setminus\{\alpha_i\}$ is simple (and has $n_k=1$), after Lemma \ref{actionwionroots}. In both cases, the above quantity is at most $1$, finishing the proof.
\qed

\begin{rem}
The proof of $\mathrm{c})$ above gives a concrete way to find, given $\lambda\in\mathcal{A}_0\setminus F_{P^\vee}$, an element $\omega_i\in\Omega$ such that $\lambda\in \omega_i F_{P^\vee}$. Namely, it is the element corresponding to the minuscule coweight $\varpi_i^\vee$ with index $i$ such that $\left<\lambda,\alpha_i\right>=\max_{j\in J}\left<\lambda,\alpha_j\right>$.
\end{rem}

\begin{cor}\label{funddomforextendedaffineweylgroup}
The convex polytope of $V^*$ defined by
\begin{align*}
F_{P^\vee}:&=\{\lambda\in{\mathcal{A}_0}~;~\left<\lambda,\alpha_0+\alpha\right>\le1,~\forall\alpha\in\Pi~;~n_\alpha=1\} \\ 
&=\{\lambda\in V^*~;~\left<\lambda,\alpha_0\right>\le1,~\forall \alpha\in\Pi,~\left<\lambda,\alpha\right>\ge0~~\text{and}~~n_\alpha=1~\Rightarrow~\left<\lambda,\alpha_0+\alpha\right>\le1\}
\end{align*}
is a fundamental domain for the extended affine Weyl group $\widehat{W_\mathrm{a}}$. Moreover, if $\lambda\in F_{P^\vee}$ is an interior point and $1\ne\widehat{w}\in\widehat{W_{\rm a}}$, then $\widehat{w}(\lambda)\notin F_{P^\vee}$.
\end{cor}
\begin{proof}
This is obvious using the Proposition \ref{funddomforOmega}, the fact that ${\mathcal{A}_0}$ is a fundamental domain for $W_\mathrm{a}$ and since $\widehat{W_\mathrm{a}}=W_\mathrm{a}\rtimes\Omega$.
\end{proof}

Combining the Proposition \ref{funddomforOmega} and the Lemmae \ref{Omegaactsonvertices}, \ref{actionwimonroots} and \ref{actionwionroots}, we obtain an explicit expression for the partial action of $\Omega$ on $F_{P^\vee}$, in terms of permutations of the nodes of the extended Dynkin diagram, as detailed in the Table \ref{extendeddynkindiagrams}. More precisely, we have the following result:
\begin{cor}[{\cite[Theorem, (4)]{komrakov-premet}}]\label{explicit_action}
Let $\lambda=\sum_{i=1}^r\lambda_i\varpi_i^\vee\in F_{P^\vee}$ and $j\in J$ such that $\left<\lambda,\alpha_j+\alpha_0\right>=1$. Then, we have
\begin{equation}\label{explicit_action_eq}
\omega_j^{-1}(\lambda)=\lambda_{j}\varpi^\vee_{\sigma_j^{-1}(0)}+\sum_{\sigma_j(i)\ne0}\lambda_{\sigma_j(i)}\varpi_i^\vee,
\end{equation}
where $\sigma_j$ is the permutation of the nodes $\{0,1,\dotsc,r\}$ of the extended Dynkin diagram of $\Phi$ (see Table \ref{extendeddynkindiagrams}) corresponding to the action of $\omega_j$ on the vertices of $\mathcal{A}_0$.
\end{cor}
\begin{proof}
Let us write $\omega_j^{-1}(\lambda)=\sum_{i=1}^r\mu_i\varpi_i^\vee$. For any $1\le i \le r$, we have
\begin{align*}
\mu_i&=\left<\omega_j^{-1}(\lambda),\alpha_i\right>=\left<\lambda-\varpi_j^\vee,w_j\alpha_i\right>=\left\{\begin{array}{ll}1-\left<\lambda,\alpha_0\right> & \text{if }\omega_j(\varpi_i^\vee)=0,\\[.5em] \left<\lambda,\alpha_k\right> & \text{if }\omega_j\left(\tfrac{\varpi_i^\vee}{n_i}\right)=\tfrac{\varpi_k^\vee}{n_k}.\end{array}\right. \\
&=\left\{\begin{array}{ll}\lambda_j & \text{if }\sigma_j(i)=0, \\[.5em] \lambda_{\sigma_j(i)} & \text{otherwise.}\end{array}\right.
\end{align*}
\end{proof}

\begin{rem}
In the table of \cite{komrakov-premet}, the elements $\sigma_i$ are expressed as permutations of the subset $\{1,\dotsc,r\}$ of simple roots, forgetting the action on $0$ and its preimage. This convention has the advantage of giving a more compact expression for the partial action of $\Omega$ on $F_{P^\vee}$. Indeed, keeping the notation of the Corollary \ref{explicit_action}, the convention of Komrakov--Premet for $\sigma_i$ permits to rewrite \eqref{explicit_action_eq} as $\omega_j^{-1}(\lambda)=\sum_{i=1}^r\lambda_{\sigma_j(i)}\varpi_i^\vee$. However, our approach allows to identify $\Omega$ with a subgroup of $\Sym(\{0,\dotsc,r\})$, formed by automorphisms of the extended diagram, as explained in the Corollary \ref{pi(Omega)=Aut(Dyn0)}.
\end{rem}

We can now investigate the injectivity condition. We have the following result:
\begin{lem}\label{A0inter(A0+P)=0}
The hypothesis of the Lemma \ref{fromlatticetotorus} is fulfilled in the adjoint case. In other words, one has
\[{F_{P^\vee}}\cap(P^\vee\setminus\{0\}+{F_{P^\vee}})=\emptyset.\]
\end{lem}
\begin{proof}
Let $\lambda,\lambda'\in F_{P^\vee}\subset{\mathcal{A}_0}$ such that $\mu:=\lambda-\lambda'\in P^\vee$. We have to show that $\mu=0$. First observe that $-1\le\left<\mu,\alpha\right>\le1$ for any root $\alpha\in\Phi$.

Suppose that $\mu\notin{\mathcal{A}_0}$. Since $\left<\mu,\alpha_0\right>\le1$, there must be some $i$ such that $-1\le\left<\mu,\alpha_i\right><0$. But $\left<\mu,\alpha_i\right>\in\Z$ since we have assumed that $\mu\in P^\vee$, so we get $\left<\mu,\alpha_i\right>=-1$. Thus, we compute $1\ge\left<\lambda',\alpha_i\right>=\left<\lambda,\alpha_i\right>+1\ge1$ and hence, $\left<\lambda',\alpha_i\right>=1$. Furthermore, we have
\[1\ge\left<\lambda',\alpha_0\right>=\sum_j n_j\left<\lambda',\alpha_j\right>=n_i+\sum_{j\ne i}n_j\underbrace{\left<\lambda',\alpha_j\right>}_{\ge0}\ge n_i\ge1~\Rightarrow~\left<\lambda',\alpha_0\right>=n_i=1\]
and thus, $\left<\lambda',\alpha_0+\alpha_i\right>=2$, a contradiction since $\lambda'\in F_{P^\vee}$.

We have shown that $\mu\in{\mathcal{A}_0}\cap P^\vee$. By \cite[VI, \S 2.2, Prop. 5]{bourbaki456}, if $\mu\ne0$, then there is some $k\in J$ such that $\mu=\varpi_k^\vee$. Since $\lambda\in F_{P^\vee}$, we have $1\ge\left<\lambda,\alpha_0+\alpha_k\right>=\underbrace{\left<\lambda',\alpha_0+\alpha_k\right>}_{\ge0}+2\ge2$, which is absurd.
\end{proof}

\begin{center}
\resizebox{0.99\textwidth}{!}{
\begin{tabular}{|c|c|c|c|}
\hline
Type & Extended Dynkin diagram & Fundamental group $\Omega\simeq P/Q$ & Permutations $\sigma_i$ representing non-trivial elements $\omega_i$ of $\Omega\le\mathrm{Aut}(\widehat{\mathcal{D}})$ \\
\hline
\hline
$\widetilde{A_1}$ & \begin{minipage}{0.5\textwidth}\centering \begin{tikzpicture}
	\coordinate (a) at (0,0);
	\coordinate (b) at (1,0);
	
	\draw (a) node[below]{$1$};
	\draw (b) node[below]{$0$};
	
	\draw (a)--(b);
	
	\fill[fill=white] (a) circle (2.5pt);
	\fill[fill=white] (b) circle (2.5pt);
	\node[mark size=2.5pt] at (b) {\pgfuseplotmark{otimes}};
	\draw (a) circle (2.5pt);
	
	\draw (1/2,0) node[above]{$\infty$};
\end{tikzpicture}\end{minipage} & $\Z/2\Z$ & $\sigma_1=(0,1)$ \\
\hline
 $\widetilde{A_n}~(n\ge2)$ & \begin{minipage}{0.5\textwidth}\centering \begin{tikzpicture}[scale=1]
	\coordinate (a) at (-2,0);
	\coordinate (b) at (-1,0);
	\coordinate (c) at (0,0);
	\coordinate (d) at (1,0);
	\coordinate (e) at (2,0);
	\coordinate (f) at (0,1);
	\coordinate (bc) at (-1/2,0);
	\coordinate (cd) at (1/2,0);
	
	\draw (a) node[below]{$1$};
	\draw (b) node[below]{$2$};
	\draw (c) node{$\cdots$};
	\draw (d) node[below]{${n-1}$};
	\draw (e) node[below]{$n$};
	\draw (f) node[above]{$0$};
	
	\draw (a)--(b)
	(b)--(bc)
	(cd)--(d)
	(d)--(e)
	(e)--(f)
	(f)--(a);
	
	\fill[fill=white] (a) circle (2.5pt);
	\fill[fill=white] (b) circle (2.5pt);
	\fill[fill=white] (d) circle (2.5pt);
	\fill[fill=white] (e) circle (2.5pt);
	\fill[fill=white] (f) circle (2.5pt);
	
	\draw (a) circle (2.5pt);
	\draw (b) circle (2.5pt);
	\draw (d) circle (2.5pt);
	\draw (e) circle (2.5pt);
	\node[mark size=2.5pt] at (f) {\pgfuseplotmark{otimes}};
\end{tikzpicture}\end{minipage} & $\Z/(n+1)\Z$ & $\begin{array}{ll} \sigma_{1}=(0,1,2,\cdots,n) \\ \\ \sigma_{i}=(\sigma_{1})^i,~1\le i \le n\end{array}$ \\
\hline
 & & & \\
$\widetilde{B_2}=\widetilde{C_2}$ & \begin{minipage}{0.5\textwidth}\centering \begin{tikzpicture}[scale=1]
	\coordinate (a) at (0,0);
	\coordinate (e) at (1,0);
	\coordinate (f) at (2,0);
	
	\draw (a) node[below]{$0$};
	\draw (e) node[below]{$1$};
	\draw (f) node[below]{$2$};
	
	\draw[decoration={markings, mark=at position 0.75 with {\arrow{angle 60}}},postaction={decorate},double distance=2.5pt] (a)--(e);
	\draw[decoration={markings, mark=at position 0.75 with {\arrow{angle 60}}},postaction={decorate},double distance=2.5pt] (f)--(e);
	
	\fill[fill=white] (a) circle (2.5pt);
	\node[mark size=2.5pt] at (a) {\pgfuseplotmark{otimes}};
	\fill[fill=black] (e) circle (2.5pt);
	\fill[fill=white] (f) circle (2.5pt);
    \draw (f) circle (2.5pt);
\end{tikzpicture}\end{minipage} & $\Z/2\Z$ & $\sigma_{1}=(0,2)$ \\
 & & & \\
\hline
$\widetilde{B_n}~(n\ge3)$ & \begin{minipage}{0.5\textwidth}\centering \begin{tikzpicture}[scale=1]
	\coordinate (a) at (-1/2,1.732/2);
	\coordinate (b) at (-1/2,-1.732/2);
	\coordinate (c) at (1/2,0);
	\coordinate (d) at (1+1/2,0);
	\coordinate (e) at (2+1/2,0);
	\coordinate (f) at (3+1/2,0);
	\coordinate (g) at (4+1/2,0);
	
	\draw (a) node[left]{$1$};
	\draw (b) node[left]{$0$};
	\draw (c) node[below]{$2$};
	\draw (d) node[below]{$3$};
	\draw (e) node{$\cdots$};
	\draw (f) node[below]{${n-1}$};
	\draw (g) node[below]{$n$};
	
	\draw (a)--(c)
	(b)--(c)
	(c)--(d)
	(d)--(2,0)
	(3,0)--(f);
	\draw[decoration={markings, mark=at position 0.75 with {\arrow{angle 60}}},postaction={decorate},double distance=2.5pt] (f)--(g);
	
	\fill[fill=white] (a) circle (2.5pt);
	\fill[fill=black] (c) circle (2.5pt);
	\fill[fill=black] (d) circle (2.5pt);
	\fill[fill=black] (f) circle (2.5pt);
	\fill[fill=black] (g) circle (2.5pt);
	\fill[fill=white] (b) circle (2.5pt);
	\node[mark size=2.5pt] at (b) {\pgfuseplotmark{otimes}};
	\draw (a) circle (2.5pt);
\end{tikzpicture}\end{minipage} & $\Z/2\Z$ & $\sigma_{1}=(0,1)$ \\
\hline
 & & & \\
$\widetilde{C_n}~(n\ge3)$ & \begin{minipage}{0.5\textwidth}\centering \begin{tikzpicture}[scale=1]
	\coordinate (a) at (-3,0);
	\coordinate (b) at (-2,0);
	\coordinate (c) at (-1,0);
	\coordinate (d) at (0,0);
	\coordinate (e) at (1,0);
	\coordinate (f) at (2,0);
	
	\draw (a) node[below]{$0$};
	\draw (b) node[below]{$1$};
	\draw (c) node[below]{$2$};
	\draw (d) node{$\cdots$};
	\draw (e) node[below]{${n-1}$};
	\draw (f) node[below]{$n$};
	
	\draw[decoration={markings, mark=at position 0.75 with {\arrow{angle 60}}},postaction={decorate},double distance=2.5pt] (a)--(b);
	\draw[decoration={markings, mark=at position 0.75 with {\arrow{angle 60}}},postaction={decorate},double distance=2.5pt] (f)--(e);
	\draw (b)--(c)
	(c)--(-1/2,0)
	(1/2,0)--(e);
	
	\fill[fill=white] (a) circle (2.5pt);
	\node[mark size=2.5pt] at (a) {\pgfuseplotmark{otimes}};
	\fill[fill=black] (b) circle (2.5pt);
	\fill[fill=black] (c) circle (2.5pt);
	\fill[fill=black] (e) circle (2.5pt);
	\fill[fill=white] (f) circle (2.5pt);
	
    \draw (f) circle (2.5pt);
\end{tikzpicture}\end{minipage} & $\Z/2\Z$ & $\displaystyle{\sigma_{n}=(0,n)\prod_{i=1}^{\left\lfloor \frac{n-1}{2}\right\rfloor}(i,{n-i})}$ \\
 & & & \\
\hline
 & & & \\
$\widetilde{D_{2n}}~(n\ge2)$ & \begin{minipage}{0.5\textwidth}\centering \begin{tikzpicture}[scale=1]
	\coordinate (a) at (-1/2,1.732/2);
	\coordinate (b) at (-1/2,-1.732/2);
	\coordinate (c) at (1/2,0);
	\coordinate (d) at (1+1/2,0);
	\coordinate (e) at (2+1/2,0);
	\coordinate (f) at (3+1/2,0);
	\coordinate (g) at (4+1/2,1.732/2);
	\coordinate (h) at (4+1/2,-1.732/2);
	
	\draw (a) node[left]{$1$};
	\draw (b) node[left]{$0$};
	\draw (c) node[below]{$2$};
	\draw (d) node[below]{$3$};
	\draw (e) node{$\cdots$};
	\draw (f) node[right]{${2n-2}$};
	\draw (g) node[right]{${2n}$};
	\draw (h) node[right]{${2n-1}$};
	
	\draw (a)--(c)
	(b)--(c)
	(c)--(d)
	(d)--(2,0)
	(3,0)--(f)
	(f)--(g)
	(f)--(h);
	
	\fill[fill=white] (a) circle (2.5pt);
	\fill[fill=black] (c) circle (2.5pt);
	\fill[fill=black] (d) circle (2.5pt);
	\fill[fill=black] (f) circle (2.5pt);
	\fill[fill=white] (g) circle (2.5pt);
	\fill[fill=white] (b) circle (2.5pt);
	\fill[fill=white] (h) circle (2.5pt);
	\node[mark size=2.5pt] at (b) {\pgfuseplotmark{otimes}};
	\draw (a) circle (2.5pt);
	\draw (g) circle (2.5pt);
	\draw (h) circle (2.5pt);
\end{tikzpicture}\end{minipage} & $\Z/2\Z\oplus\Z/2\Z$ & $\displaystyle{\begin{array}{llll} \sigma_{1}=(0,1)({2n-1},{2n}) \\ \\ \sigma_{{2n-1}}=(0,{2n-1})(1,{2n})\prod_{i=2}^{n-1}(i,{2n-i}) \\ \\ \sigma_{{2n}}=(0,{2n})(1,{2n-1})\prod_{i=2}^{n-1}(i,{2n-i})=\sigma_{1}\sigma_{{2n-1}}\end{array}}$ \\
 & & & \\
\hline
 & & & \\
$\widetilde{D_{2n+1}}~(n\ge2)$ & \begin{minipage}{0.5\textwidth}\centering \begin{tikzpicture}[scale=1]
	\coordinate (a) at (-1/2,1.732/2);
	\coordinate (b) at (-1/2,-1.732/2);
	\coordinate (c) at (1/2,0);
	\coordinate (d) at (1+1/2,0);
	\coordinate (e) at (2+1/2,0);
	\coordinate (f) at (3+1/2,0);
	\coordinate (g) at (4+1/2,1.732/2);
	\coordinate (h) at (4+1/2,-1.732/2);
	
	\draw (a) node[left]{$1$};
	\draw (b) node[left]{$0$};
	\draw (c) node[below]{$2$};
	\draw (d) node[below]{$3$};
	\draw (e) node{$\cdots$};
	\draw (f) node[right]{${2n-1}$};
	\draw (g) node[right]{${2n+1}$};
	\draw (h) node[right]{${2n}$};
	
	\draw (a)--(c)
	(b)--(c)
	(c)--(d)
	(d)--(2,0)
	(3,0)--(f)
	(f)--(g)
	(f)--(h);
	
	\fill[fill=white] (a) circle (2.5pt);
	\fill[fill=black] (c) circle (2.5pt);
	\fill[fill=black] (d) circle (2.5pt);
	\fill[fill=black] (f) circle (2.5pt);
	\fill[fill=white] (g) circle (2.5pt);
	\fill[fill=white] (b) circle (2.5pt);
	\fill[fill=white] (h) circle (2.5pt);
	\node[mark size=2.5pt] at (b) {\pgfuseplotmark{otimes}};
	\draw (a) circle (2.5pt);
	\draw (g) circle (2.5pt);
	\draw (h) circle (2.5pt);
\end{tikzpicture}\end{minipage} & $\Z/4\Z$ & $\displaystyle{\begin{array}{llll} \sigma_{1}=(0,1)({2n},{2n+1}) \\ \\ \sigma_{{2n}}=(0,{2n},1,{2n+1})\prod_{i=2}^{n}(i,{2n+1-i}) \\ \\ \sigma_{{2n+1}}=(0,{2n+1},1,{2n})\prod_{i=2}^{n}(i,{2n+1-i})\end{array}}$ \\
 & & & \\
\hline
$\widetilde{E_6}$ & \begin{minipage}{0.5\textwidth}\centering \begin{tikzpicture}
	\coordinate (a) at (-2,0);
	\coordinate (b) at (-1,0);
	\coordinate (c) at (0,0);
	\coordinate (d) at (1,0);
	\coordinate (e) at (2,0);
	\coordinate (f) at (0,1);
	\coordinate (g) at (0,2);
	
	\draw (a) node[below]{$1$};
	\draw (b) node[below]{$3$};
	\draw (c) node[below]{$4$};
	\draw (d) node[below]{$5$};
	\draw (e) node[below]{$6$};
	\draw (f) node[right]{$2$};
	\draw (g) node[right]{$0$};
	
	\draw (a)--(b)
	(b)--(c)
	(c)--(d)
	(d)--(e)
	(c)--(f)
	(f)--(g);
	
	\fill[fill=white] (a) circle (2.5pt);
	\fill[fill=white] (e) circle (2.5pt);
	\fill[fill=white] (g) circle (2.5pt);
	\fill[fill=black] (b) circle (2.5pt);
	\fill[fill=black] (c) circle (2.5pt);
	\fill[fill=black] (d) circle (2.5pt);
	\fill[fill=black] (f) circle (2.5pt);
	\draw (a) circle (2.5pt);
	\draw (e) circle (2.5pt);
	\node[mark size=2.5pt] at (g) {\pgfuseplotmark{otimes}};
\end{tikzpicture}\end{minipage} & $\Z/3\Z$ & $\begin{array}{ll} \sigma_{1}=(0,1,6)(2,3,5) \\ \\ \sigma_{6}=(1,0,6)(3,2,5)=\sigma_{1}^{-1}\end{array}$ \\
\hline
$\widetilde{E_7}$ & \begin{minipage}{0.5\textwidth}\centering \begin{tikzpicture}
	\coordinate (a) at (-2,0);
	\coordinate (b) at (-1,0);
	\coordinate (c) at (0,0);
	\coordinate (d) at (1,0);
	\coordinate (e) at (2,0);
	\coordinate (f) at (0,1);
	\coordinate (g) at (3,0);
	\coordinate (h) at (-3,0);
	
	\draw (a) node[below]{$1$};
	\draw (b) node[below]{$3$};
	\draw (c) node[below]{$4$};
	\draw (d) node[below]{$5$};
	\draw (e) node[below]{$6$};
	\draw (f) node[right]{$2$};
	\draw (g) node[below]{$7$};
	\draw (h) node[below]{$0$};
	
	\draw (a)--(b)
	(b)--(c)
	(c)--(d)
	(d)--(e)
	(c)--(f)
	(e)--(g)
	(h)--(a);
	
	\fill[fill=black] (a) circle (2.5pt);
	\fill[fill=black] (e) circle (2.5pt);
	\fill[fill=white] (g) circle (2.5pt);
	\fill[fill=white] (h) circle (2.5pt);
	\fill[fill=black] (b) circle (2.5pt);
	\fill[fill=black] (c) circle (2.5pt);
	\fill[fill=black] (d) circle (2.5pt);
	\fill[fill=black] (f) circle (2.5pt);
	\draw (g) circle (2.5pt);
	\node[mark size=2.5pt] at (h) {\pgfuseplotmark{otimes}};
\end{tikzpicture}\end{minipage} & $\Z/2\Z$ & $\sigma_{7}=(0,7)(1,6)(3,5)$ \\
\hline
$\widetilde{E_8}$ & \begin{minipage}{0.5\textwidth}\centering \begin{tikzpicture}
	\coordinate (a) at (-2,0);
	\coordinate (b) at (-1,0);
	\coordinate (c) at (0,0);
	\coordinate (d) at (1,0);
	\coordinate (e) at (2,0);
	\coordinate (f) at (0,1);
	\coordinate (g) at (3,0);
	\coordinate (h) at (5,0);
	\coordinate (i) at (4,0);
	
	\draw (a) node[below]{$1$};
	\draw (b) node[below]{$3$};
	\draw (c) node[below]{$4$};
	\draw (d) node[below]{$5$};
	\draw (e) node[below]{$6$};
	\draw (f) node[right]{$2$};
	\draw (g) node[below]{$7$};
	\draw (h) node[below]{$0$};
	\draw (i) node[below]{$8$};
	
	\draw (a)--(b)
	(b)--(c)
	(c)--(d)
	(d)--(e)
	(c)--(f)
	(e)--(g)
	(h)--(i)
	(g)--(i);
	
	\fill[fill=black] (a) circle (2.5pt);
	\fill[fill=black] (e) circle (2.5pt);
	\fill[fill=black] (g) circle (2.5pt);
	\fill[fill=white] (h) circle (2.5pt);
	\fill[fill=black] (b) circle (2.5pt);
	\fill[fill=black] (c) circle (2.5pt);
	\fill[fill=black] (d) circle (2.5pt);
	\fill[fill=black] (f) circle (2.5pt);
	\fill[fill=black] (i) circle (2.5pt);
	\node[mark size=2.5pt] at (h) {\pgfuseplotmark{otimes}};
\end{tikzpicture}\end{minipage} & $1$ & $\varnothing$ \\
\hline
 & & & \\
\centering{$\widetilde{F_4}$} & \begin{minipage}{0.5\textwidth}\centering \begin{tikzpicture}
	\coordinate (a) at (-2,0);
	\coordinate (b) at (-1,0);
	\coordinate (c) at (0,0);
	\coordinate (d) at (1,0);
	\coordinate (e) at (2,0);
	
	\draw (a) node[below]{$0$};
	\draw (b) node[below]{$1$};
	\draw (c) node[below]{$2$};
	\draw (d) node[below]{$3$};
	\draw (e) node[below]{$4$};
	
	\draw (a)--(b)
	(b)--(c)
	(d)--(e);	
	\draw[decoration={markings, mark=at position 0.75 with {\arrow{angle 60}}},postaction={decorate},double distance=2.5pt] (c)--(d);
	
	\fill[fill=white] (a) circle (2.5pt);
	\node[mark size=2.5pt] at (a) {\pgfuseplotmark{otimes}};
	\fill[fill=black] (b) circle (2.5pt);
	\fill[fill=black] (c) circle (2.5pt);
	\fill[fill=black] (d) circle (2.5pt);
	\fill[fill=black] (e) circle (2.5pt);
\end{tikzpicture}\end{minipage} & $1$ & $\varnothing$ \\
\hline
 & & & \\
$\widetilde{G_2}$ & \begin{minipage}{0.5\textwidth}\centering \begin{tikzpicture}
	\coordinate (a) at (-1,0);
	\coordinate (b) at (0,0);
	\coordinate (c) at (1,0);
	
	\draw (a) node[below]{$1$};
	\draw (b) node[below]{$2$};
	\draw (c) node[below]{$0$};
	
	\draw (b)--(c);
	\draw[decoration={markings, mark=at position 0.75 with {\arrow{angle 60}}},postaction={decorate},double distance=2.5pt] (b)--(a);
	\draw (b)--(a);
	
	\fill[fill=white] (c) circle (2.5pt);
	\node[mark size=2.5pt] at (c) {\pgfuseplotmark{otimes}};
	\fill[fill=black] (a) circle (2.5pt);
	\fill[fill=black] (b) circle (2.5pt);
\end{tikzpicture}\end{minipage} & $1$ & $\varnothing$ \\
\hline
\end{tabular}}
\captionof{table}{Extended Dynkin diagrams and fundamental groups elements, represented as permutations of the nodes.}
\label{extendeddynkindiagrams}
\end{center}

\begin{cor}\label{funddomforKadjointcase}
If $K$ is an adjoint group, then $\exp(F_{P^\vee})$ is homeomorphic to the polytope $F_{P^\vee}$ and is a fundamental domain for the action of $W$ on $T$. 
\end{cor}

The Lemmae \ref{actionwimonroots} and \ref{actionwionroots} may be used to give a useful characterization of the element $w_i=w^i_0w_0$ from the Proposition \ref{descriptionOmega}.
\begin{prop}\label{characterizationwi}
For any $i\in J$, the element $w_i$ is the unique element $w\in W$ of minimal length such that 
\[\varpi_i^\vee\in ww_0({\mathcal{C}_0}),\]
where $w_0$ is the longest element of $W$.
\end{prop}
\begin{proof}
Let $w\in W$ such that $w^{-1}(\varpi_i^\vee)\in w_0({\mathcal{C}_0})=-{\mathcal{C}_0}$ with $\ell(w)$ minimal with respect to this property. First, we prove that $\mathrm{t}_{\varpi_i^\vee}w\in\Omega$, that is, $w({\mathcal{A}_0})+\varpi_i^\vee={\mathcal{A}_0}$. 

Since ${\mathcal{A}_0}=\conv(\{0\}\cup \{\varpi_j^\vee/n_j\}_{j\in I})$, it is sufficient to prove that for any $j\in I$, we have $w\left(\frac{\varpi_j^\vee}{n_j}\right)+\varpi_i^\vee\in{\mathcal{A}_0}$. Recall that $W$ acts on the set $\{\lambda\in V^*~;~\forall \alpha\in\Phi,~-1\le \left<\lambda,\alpha\right>\le 1\}$. This, together with the fact that $-w^{-1}(\varpi_i^\vee)\in{\mathcal{C}_0}$, implies that $-w^{-1}(\varpi_i^\vee)\in{\mathcal{A}_0}$. Thus, we have $-w^{-1}(\varpi_i^\vee)\in{\mathcal{A}_0}\cap P^\vee$ and so there is some $k\in J$ such that $w^{-1}(\varpi_i^\vee)=-\varpi_k^\vee$ (see \cite[VI, \S 2.2, Proposition 5]{bourbaki456}).

\underline{Claim 1} : For $j\ne i$, we have $w(\varpi_j^\vee)\ne\varpi_j^\vee$.

Suppose the contrary, then 
\[w^{-1}\left(\frac{\varpi_j^\vee}{n_j}+\varpi_i^\vee\right)=\frac{\varpi_j^\vee}{n_j}-\varpi_k^\vee\]
and by the Lemma \ref{diffovA0elementsarelinkedbyW}, there exists some $v\in W$ such that $vw^{-1}\left(\frac{\varpi_j^\vee}{n_j}+\varpi_i^\vee\right)\in{\mathcal{A}_0}$ and thus
\[1\ge\left<vw^{-1}\left(\frac{\varpi_j^\vee}{n_j}+\varpi_i^\vee\right),\underbrace{vw^{-1}(\alpha_0)}_{\in\Phi}\right>=\left<\frac{\varpi_j^\vee}{n_j}+\varpi_i^\vee,\alpha_0\right>=2,\]
a contradiction. 

\underline{Claim 2} : For $j\ne i$, we have $\left<\varpi_j^\vee,w(\alpha_k)\right>\ne0$.

Indeed, for each $k'\ne k$, we have $(ws_{k'})^{-1}(\varpi_i^\vee)=-s_{k'}(\varpi_k^\vee)=-\varpi_k^\vee\in -{\mathcal{C}_0}$ and because $\ell(w)$ is minimal, this implies that $\ell(ws_{k'})=\ell(w)+1$ and thus $w(\alpha_{k'})\in\Phi^+$. This yields $\left<w^{-1}(\varpi_j^\vee),\alpha_{k'}\right>=\left<\varpi_j^\vee,w(\alpha_{k'})\right>\ge0$ for each $k'\ne k$. If $\left<\varpi_j^\vee,w(\alpha_k)\right>=0$, then $w^{-1}(\varpi_j^\vee)\in{\mathcal{C}_0}$ and so $w^{-1}\left(\frac{\varpi_j^\vee}{n_j}\right)\in{\mathcal{A}_0}\cap P^\vee$ and we may choose $k'\ne k$ in $J$ such that $w^{-1}\left(\frac{\varpi_j^\vee}{n_j}\right)=\varpi_{k'}^\vee$. Taking the pairing of this equation against $w^{-1}(\alpha_j)$ yields $1/n_j\in\Z$, i.e. $n_j=1$ and $w^{-1}(\varpi_j^\vee)=\varpi_{k'}^\vee$. Now, the action of $W$ on $P^\vee/Q^\vee$ being trivial, we have that $\varpi_j^\vee \equiv \varpi_{k'}^\vee \pmod {Q^\vee}$ and because $n_j=n_{k'}=1$, this implies $k'=j$ and thus $w^{-1}(\varpi_j^\vee)=\varpi_j^\vee$. The Claim 1 forbids this, so our assumption that $\left<\varpi_j^\vee,w(\alpha_k)\right>=0$ was wrong.

\underline{Claim 3} : We have $w(\alpha_k)=-\alpha_0$.

Indeed, we have seen that $w(\alpha_{k'})\in\Phi^+$ for $k'\ne k$ and because $w\ne 1$, this yields $w(\alpha_k)\in\Phi^-$. Now, for $j\ne i$, we have
\[0\ge \left<\varpi_j^\vee,w(\alpha_k)\right>=n_j\left<w^{-1}\left(\frac{\varpi_j^\vee}{n_j}\right),\alpha_k\right>\]
and $\left<w^{-1}(\varpi_j^\vee),\alpha_k\right>\ne0$ by the Claim 2, so $-1\le \left<w^{-1}\left(\frac{\varpi_j^\vee}{n_j}\right),\alpha_k\right><0$ and it is an integer, so it equals $-1$ and thus 
\[\forall j \ne i,~\left<\varpi_j^\vee,w(\alpha_k)\right>=-n_j\]
and we already have $\left<\varpi_i,w(\alpha_k)\right>=\left<w^{-1}(\varpi_i^\vee),\alpha_k\right>=-1=-n_i$. This gives that $w(\alpha_k)=-\alpha_0$, as claimed.

Let $j\in I$. For $j'\ne i$, we compute
\[\left<w\left(\frac{\varpi_j^\vee}{n_j}\right)+\varpi_i^\vee,\alpha_{j'}\right>=\left<w\left(\frac{\varpi_j^\vee}{n_j}\right),\alpha_{j'}\right>\]
and this is non-negative because $w^{-1}(\alpha_{j'})\in\Phi^+$. This holds because we must have $\ell(w^{-1}s_{j'})=\ell(s_{j'}w)=\ell(w)+1$ since $(s_{j'}w)^{-1}(\varpi_i^\vee)=w^{-1}(\varpi_i^\vee)\in -{\mathcal{C}_0}$ and $\ell(w)$ is minimal with respect to this property. On the other hand, we have
\[\left<w\left(\frac{\varpi_j^\vee}{n_j}\right)+\varpi_i^\vee,\alpha_i\right>=1+\underbrace{\left<w\left(\frac{\varpi_j^\vee}{n_j}\right),\alpha_i\right>}_{\ge -1}\ge 0.\]
This proves that
\[w\left(\frac{\varpi_j^\vee}{n_j}\right)+\varpi_i^\vee\in{\mathcal{C}_0}.\]
Moreover, using the Claim 3 we get
\[\left<w\left(\frac{\varpi_j^\vee}{n_j}\right)+\varpi_i^\vee,\alpha_0\right>=1+\left<\frac{\varpi_j^\vee}{n_j},w^{-1}(\alpha_0)\right>=1-\left<\frac{\varpi_j^\vee}{n_j},\alpha_k\right>\le1\]
and thus $w\left(\frac{\varpi_j^\vee}{n_j}\right)+\varpi_i^\vee\in{\mathcal{A}_0}$, and $\mathrm{t}_{\varpi_i^\vee}w\in\Omega$ as required.

Now, since $\mathrm{t}_{\varpi_i^\vee}w\in\Omega$, the Lemmae \ref{actionwimonroots} and \ref{actionwionroots} are valid for $w$ and let $v:=ww_i^{-1}$ where $w_i=w_0^iw_0$ is the element from the Proposition \ref{descriptionOmega}. By these Lemmae, we have $v(\alpha_i)=-w(\alpha_0)=\alpha_i$ and, if $j\ne i$, we have $v(\alpha_j)\in\Pi$ except in the case where $v(\alpha_j)=-\alpha_0$, so $w_i^{-1}(\alpha_j)=-w^{-1}(\alpha_0)=\alpha_k$ with $k$ such that $-w^{-1}(\varpi_i^\vee)=\varpi_k^\vee=-w_i^{-1}(\varpi_i^\vee)$. In this case, we have $\alpha_j=w_i(\alpha_k)=-\alpha_0$, which is excluded. Hence we get $v(\Pi)\subset\Phi^+$ and $v=1$.
\end{proof}

\begin{rem}\label{extensionofOmegaelementstoP}
The above characterization may also be used to generalize the construction of the $\omega_i$'s to $P^\vee$. More precisely, for any $\lambda\in P^\vee$, consider the element $w_\lambda$ with minimal length among those $w\in W$ such that $\lambda\in ww_0({\mathcal{C}_0})$ and define $u_\lambda:=\mathrm{t}_\lambda w_\lambda\in \widehat{W_\mathrm{a}}$. Note that for $i\in J$, we have $\omega_i=u_{\varpi_i^\vee}\in\Omega$. The assignment $\lambda\mapsto u_\lambda$ results in a well-defined map
\[P^\vee \stackrel{\tiny{u}}\longto \widehat{W_\mathrm{a}}.\]
\end{rem}

We finish this section by giving the vertices of the polytope $F_{P^\vee}$. 
\begin{prop}\label{verticesofFP}
Let $\mathcal{B}_m$ be the set of all isobarycenters of points in $\{0\}\cup\{\varpi_j^\vee\}_{j\in J}$ with a non-zero coefficient with respect to the origin. In other words,
\[\mathcal{B}_m:=\left\{\frac{1}{|J'|+1}\sum_{j\in J'}\varpi_j^\vee~;~J'\subseteq J\right\}\]
\[=\{0\}\cup\left\{\frac{\varpi_{i_1}^\vee+\cdots+\varpi_{i_k}^\vee}{k+1}~;~1\le k \le |J|~~\text{and}~~i_j\in J,~\forall j\right\}.\]
Then the vertices of the polytope $F_{P^\vee}$ are given by
\[\vertices(F_{P^\vee})=\mathcal{B}_m\cup\left\{\frac{\varpi_i^\vee}{n_i}\right\}_{i\in I \setminus J}.\]
\end{prop}
\begin{proof}
Denote $\mathcal{V}:=\mathcal{B}_m\cup\{\varpi_i^\vee/n_i\}_{i\in I\setminus J}$. If $J=\emptyset$, then $\mathcal{B}_m=\{0\}$ and the statement is just the last corollary from \cite[VI,\S 2.2]{bourbaki456}. So we may assume that $J$ is non-empty. Denote
\[\forall i\in I,~H_i:=\{\lambda\in V^*~;~\left<\lambda,\alpha_i\right>=0\}~~\text{and}~~H_0:=\{\lambda\in V^*~;~\left<\lambda,\alpha_0\right>=1\}.\]
We also introduce the following affine hyperplanes
\[\forall j\in J,~H_j^0:=\{\lambda\in V^*~;~\left<\lambda,\alpha_0+\alpha_j\right>=1\}.\]
By construction, any facet of $F_{P^\vee}$ is of the form $F_{P^\vee}\cap H_i$ for some $0\le i \le r$ or $F_{P^\vee}\cap H_j^0$ for some $j\in J$. Therefore, any face $f$ of $F_{P^\vee}$ is of the form
\[f=F_{P^\vee}\cap\left(\bigcap_{i\in I_f}H_i\cap\bigcap_{j\in J_f}H_j^0\right)\]
for some subsets $I_f\subseteq I$ and $J_f\subseteq J$ such that $|I_f|+|J_f|=\codim_{F_{P^\vee}}(f)=r-\dim(f)$. In particular
\[\forall v\in\vertices(F_{P^\vee}),~\exists I_v\subseteq I,~\exists J_v\subseteq J~;~|I_v|+|J_v|=r~~\text{and}~~\{v\}=F_{P^\vee}\cap\left(\bigcap_{i\in I_v}H_i\cap\bigcap_{j\in J_v}H_j^0\right).\]
It is straightforward to check that any point of $\mathcal{V}$ is in at least $r$ hyperplanes among $\{H_i,H_j^0~;~i\in I,~j\in J\}$, so $\mathcal{V}\subseteq\vertices(F_{P^\vee})$.

Conversely, let $v\in\vertices(F_{P^\vee})$ and take $I_v$ and $J_v$ as above. 

Assume first that $I_v\cap J_v\ne\emptyset$ and let $k\in I_v\cap J_v$. As $\left<v,\alpha_0+\alpha_k\right>=1$ and $\left<v,\alpha_k\right>=0$, we get $\left<v,\alpha_0\right>=1$. Since $v\in F_{P^\vee}$, this implies that $\left<v,\alpha_j\right>=0$ for any $j\in J$, so $J_v=J$. If $J=I$, then $v=0\in\mathcal{V}$. Otherwise, there is some $\ell\in I\setminus J$ with $\left<v,\alpha_\ell\right>>0$, so $\ell\notin I_v$ and we have
\[\left\{\frac{\varpi_\ell^\vee}{n_\ell}\right\}=F_{P^\vee}\cap\left(\bigcap_{\ell\ne i\in I}H_i\cap\bigcap_{j\in J}H_j^0\right)\subseteq F_{P^\vee}\cap\left(\bigcap_{i\in I_v}H_i\cap\bigcap_{j\in J_v}H_j^0\right)=\{v\},\]
thus $v=\varpi_\ell^\vee/n_\ell\in\mathcal{V}$. So if $I_v\cap J_v$ is non-empty, then $v\in\mathcal{V}$.

Now, if $I_v\cap J_v=\emptyset$ we have
\[\left\{\frac{1}{|J_v|+1}\sum_{j\in J_v}\varpi_j^\vee\right\}\subseteq F_{P^\vee}\cap\left(\bigcap_{i\in I_v}H_i\cap\bigcap_{j\in J_v}H_j^0\right)=\{v\}\]
and thus $v=\frac{1}{|J_v|+1}\sum_{j\in J_v}\varpi_j^\vee\in\mathcal{V}$.
\end{proof}

\begin{exemple}\label{figures}
The Figures \ref{fig:A2}, \ref{fig:B2} and \ref{fig:B3} display the fundamental domain $F_{P^\vee}$ inside the fundamental alcove, itself inside the fundamental chamber for the respective root data of type $A_2$, $B_2$ and $B_3$.

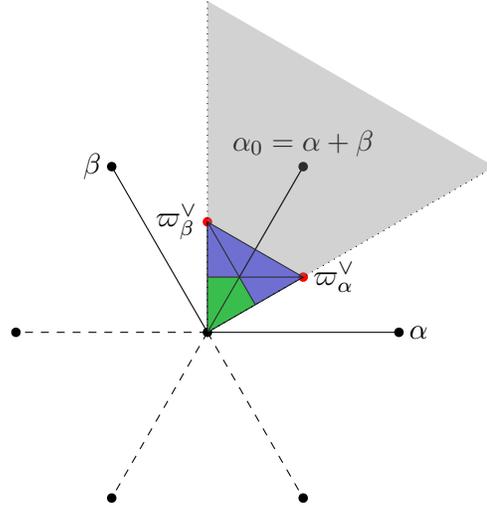
\begin{figure}[h!]
\begin{tikzpicture}[scale=1.8,rotate=45]

  \coordinate (z) at (0,0);
  \coordinate (a) at (1,-1);
  \coordinate (b) at (0.3660254040,1.366025404);
  \coordinate (ab) at (1+0.3660254040,-1+1.366025404);
  \coordinate (ma) at (-1,1);
  \coordinate (mb) at (-0.3660254040,-1.366025404);
  \coordinate (mab) at (-1-0.3660254040,1-1.366025404);
  
  \coordinate (la) at (2/3+1/3*0.3660254040,-2/3+1/3*1.366025404);
  \coordinate (lb) at (1/3+2/3*0.3660254040,-1/3+2/3*1.366025404);
  
  \coordinate (lad) at (1/3+1/6*0.3660254040,-1/3+1/6*1.3660254040);
  \coordinate (lbd) at (1/6+1/3*0.3660254040,-1/6+1/3*1.3660254040);
  \coordinate (barlab) at (1/3+1/3*0.3660254040,-1/3+1/3*1.3660254040);
  
  \fill[fill=black] (a) circle (1pt);
  \fill[fill=black] (b) circle (1pt);
  \fill[fill=black] (ab) circle (1pt);
  \fill[fill=black] (z) circle (1pt);
  \fill[fill=black] (ma) circle (1pt);
  \fill[fill=black] (mb) circle (1pt);
  \fill[fill=black] (mab) circle (1pt);
  
  \fill[fill=red] (la) circle (1pt);
  \fill[fill=red] (lb) circle (1pt);
  
  \draw (z)--(a);
  \draw (z)--(b);
  \draw (z)--(ab);
  \draw[dashed] (z)--(ma);
  \draw[dashed] (z)--(mb);
  \draw[dashed] (z)--(mab);
  
  \draw (a) node[right]{$\alpha$};
  \draw (b) node[left]{$\beta$};
  \draw (ab) node[above]{$\alpha_0=\alpha+\beta$};
  \draw (la) node[right]{$\varpi_\alpha^\vee$};
  \draw (lb) node[left]{$\varpi_\beta^\vee$};
  
  \fill[fill=blue,opacity=0.6] (z)--(la)--(lb);
  \fill[fill=green,opacity=0.8] (z)--(lad)--(barlab)--(lbd);
  
  \draw (z)--(la);
  \draw (z)--(lb);
  \draw (la)--(lb);
  \draw (la)--(lbd);
  \draw (lb)--(lad);
  \draw (z)--(ab);
  
  \coordinate (x) at (2+0.3660254040,-2+1.3660254040);
  \coordinate (y) at (1+2*0.3660254040,-1+2*1.3660254040);
  
  \fill[fill=gray,opacity=0.35] (x)--(y)--(z);
  \draw[dotted] (z)--(x);
  \draw[dotted] (z)--(y);
\end{tikzpicture}
\caption{The fundamental domain $F_{P^\vee}$ (in green) inside ${\mathcal{A}_0}$ (in blue) inside the Weyl chamber ${\mathcal{C}_0}$ (in gray) in type $A_2$.}\label{fig:A2}
\end{figure}

\begin{figure}[h!]
\begin{tikzpicture}[x={(-1cm,0cm)},y={(0cm,1cm)},scale=1.8,rotate=-90]
  \coordinate (z) at (0,0);
  \coordinate (b) at (1,-1);
  \coordinate (a) at (0,1);
  \coordinate (ab) at (1,0);
  \coordinate (dab) at (1,1);
  \coordinate (ma) at (0,-1);
  \coordinate (mb) at (-1,1);
  \coordinate (mab) at (-1,0);
  \coordinate (mdab) at (-1,-1);
  
  \coordinate (la) at (1/2,1/2);
  \coordinate (lb) at (1,0);
  \coordinate (lac) at (1,1);
  \coordinate (ldb) at (1/2,0);
  
  \fill[fill=black] (z) circle (1pt);
  \fill[fill=black] (a) circle (1pt);
  \fill[fill=black] (b) circle (1pt);
  \fill[fill=black] (ab) circle (1pt);
  \fill[fill=black] (dab) circle (1pt);
  \fill[fill=black] (ma) circle (1pt);
  \fill[fill=black] (mb) circle (1pt);
  \fill[fill=black] (mab) circle (1pt);
  \fill[fill=black] (mdab) circle (1pt);
  
  \fill[fill=red] (lac) circle (1pt);
  \fill[fill=red] (lb) circle (1pt);
  
  \draw (a) node[right]{$\alpha$};
  \draw (b) node[left]{$\beta$};
  \draw (ab) node[above]{$\alpha+\beta=\varpi_\beta^\vee$};
  \draw (dab) node[right]{$\alpha_0=2\alpha+\beta=\varpi_\alpha^\vee$};
  
  \draw (z)--(a);
  \draw (z)--(b);
  \draw (z)--(ab);
  \draw (z)--(dab);
  \draw (la)--(ldb);
  \draw (la)--(lb);
  \draw[dashed] (z)--(ma);
  \draw[dashed] (z)--(mb);
  \draw[dashed] (z)--(mab);
  \draw[dashed] (z)--(mdab);
  
  \fill[fill=blue,opacity=0.6] (z)--(la)--(lb);
  \fill[fill=green,opacity=0.8] (z)--(la)--(ldb);
  
  \coordinate (x) at (1.7,1.7);
  \coordinate (y) at (1.7,0);
  
  \fill[fill=gray,opacity=0.35] (x)--(y)--(z);
  \draw[dotted] (z)--(x);
  \draw[dotted] (z)--(y);
\end{tikzpicture}
\caption{The fundamental domain $F_{P^\vee}$ (in green) inside ${\mathcal{A}_0}$ (in blue) inside the Weyl chamber ${\mathcal{C}_0}$ (in gray) in type $B_2=C_2$.}\label{fig:B2}
\end{figure}
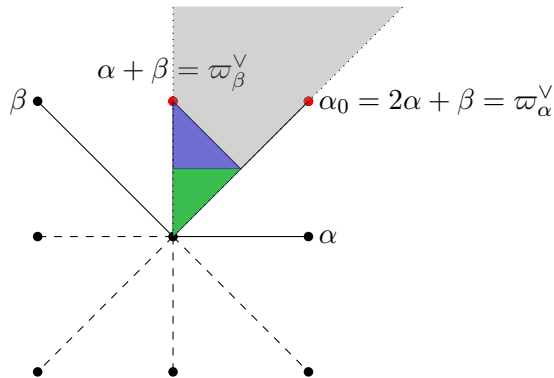

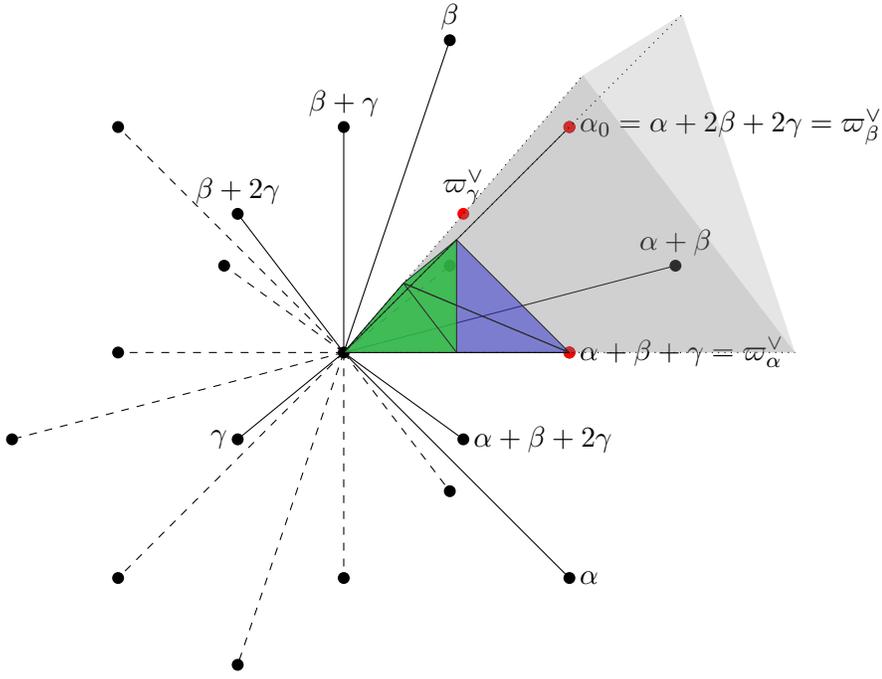
\begin{figure}[h!]
\begin{tikzpicture}[x={(1cm,0cm)},y={(0cm,1cm)},z={(-4.7mm,-3.85mm)},scale=3]
  \coordinate (z) at (0,0,0);
  \coordinate (a) at (1,-1,0);
  \coordinate (b) at (0,1,-1);
  \coordinate (c) at (0,0,1);
  \coordinate (ab) at (1,0,-1);
  \coordinate (dadbc) at (1,1,0);
  \coordinate (abdc) at (1,0,1);
  \coordinate (bdc) at (0,1,1);
  \coordinate (abc) at (1,0,0);
  \coordinate (bc) at (0,1,0);
  \coordinate (ma) at (-1,1,0);
  \coordinate (mb) at (0,-1,1);
  \coordinate (mc) at (0,0,-1);
  \coordinate (mab) at (-1,0,1);
  \coordinate (mdadbc) at (-1,-1,0);
  \coordinate (mabdc) at (-1,0,-1);
  \coordinate (mbdc) at (0,-1,-1);
  \coordinate (mabc) at (-1,0,0);
  \coordinate (mbc) at (0,-1,0);
  
  \fill[fill=black] (z) circle (0.75pt);
  \fill[fill=black] (a) circle (0.75pt);
  \fill[fill=black] (b) circle (0.75pt);
  \fill[fill=black] (c) circle (0.75pt);
  \fill[fill=black] (ab) circle (0.75pt);
  \fill[fill=black] (dadbc) circle (0.75pt);
  \fill[fill=black] (abdc) circle (0.75pt);
  \fill[fill=black] (bdc) circle (0.75pt);
  \fill[fill=black] (abc) circle (0.75pt);
  \fill[fill=black] (bc) circle (0.75pt);
  \fill[fill=black] (ma) circle (0.75pt);
  \fill[fill=black] (mb) circle (0.75pt);
  \fill[fill=black] (mc) circle (0.75pt);
  \fill[fill=black] (mab) circle (0.75pt);
  \fill[fill=black] (mdadbc) circle (0.75pt);
  \fill[fill=black] (mabdc) circle (0.75pt);
  \fill[fill=black] (mbdc) circle (0.75pt);
  \fill[fill=black] (mabc) circle (0.75pt);
  \fill[fill=black] (mbc) circle (0.75pt);
  
  \draw (z)--(a);
  \draw (z)--(b);
  \draw (z)--(c);
  \draw (z)--(ab);
  \draw (z)--(dadbc);
  \draw (z)--(abdc);
  \draw (z)--(bdc);
  \draw (z)--(abc);
  \draw (z)--(bc);
  \draw[dashed] (z)--(ma);
  \draw[dashed] (z)--(mb);
  \draw[dashed] (z)--(mc);
  \draw[dashed] (z)--(mab);
  \draw[dashed] (z)--(mdadbc);
  \draw[dashed] (z)--(mabdc);
  \draw[dashed] (z)--(mbdc);
  \draw[dashed] (z)--(mabc);
  \draw[dashed] (z)--(mbc);
  
  \draw (a) node[right]{$\alpha$};
  \draw (b) node[above]{$\beta$};
  \draw (c) node[left]{$\gamma$};
  \draw (ab) node[above]{$\alpha+\beta$};
  \draw (abdc) node[right]{$\alpha+\beta+2\gamma$};
  \draw (bdc) node[above]{$\beta+2\gamma$};
  \draw (abc) node[right]{$\alpha+\beta+\gamma=\varpi_\alpha^\vee$};
  \draw (bc) node[above]{$\beta+\gamma$};
  \draw (dadbc) node[right]{$\alpha_0=\alpha+2\beta+2\gamma=\varpi_\beta^\vee$};
  
  \coordinate (la) at (1,0,0);
  \coordinate (lb) at (1,1,0);
  \coordinate (lc) at (1,1,1);
  \coordinate (lad) at (1/2,0,0);
  \coordinate (lbd) at (1/2,1/2,0);
  \coordinate (lcd) at (1/2,1/2,1/2);
  
  \fill[fill=red] (la) circle (0.75pt);
  \fill[fill=red] (lb) circle (0.75pt);
  \fill[fill=red] (lc) circle (0.75pt);
  
  \draw (lc) node[above]{$\varpi_\gamma^\vee$};
  
  \fill[fill=blue,opacity=0.3] (z)--(la)--(lbd)--(z);
  \fill[fill=blue,opacity=0.3] (z)--(la)--(lcd)--(z);
  \fill[fill=blue,opacity=0.3] (z)--(lbd)--(lcd)--(z);
  \fill[fill=blue,opacity=0.3] (la)--(lbd)--(lcd);
  
  \fill[fill=green,opacity=0.5] (z)--(lad)--(lbd)--(z);
  \fill[fill=green,opacity=0.5] (z)--(lad)--(lcd)--(z);
  \fill[fill=green,opacity=0.5] (z)--(lbd)--(lcd)--(z);
  \fill[fill=green,opacity=0.5] (lad)--(lbd)--(lcd);
  
  \draw (z)--(lcd);
  \draw (z)--(lbd);
  \draw (z)--(la);
  \draw (la)--(lbd);
  \draw (lcd)--(la);
  \draw (lad)--(lbd);
  \draw (lbd)--(lcd);
  \draw (lcd)--(lad);
  \draw (la)--(lcd);

  \coordinate (x) at (2,2,2);
  \coordinate (y) at (1.5,1.5,0);
  \coordinate (t) at (2,0,0);
  
  \fill[fill=gray,opacity=0.2] (z)--(x)--(y);
  \fill[fill=gray,opacity=0.2] (z)--(y)--(t);
  \fill[fill=gray,opacity=0.2] (z)--(t)--(x);
  \draw[dotted] (z)--(x);
  \draw[dotted] (z)--(y);
  \draw[dotted] (z)--(t);
\end{tikzpicture}
\caption{The fundamental domain $F_{P^\vee}$ (in green) inside the fundamental alcove ${\mathcal{A}_0}$ (in blue) in type $B_3$.}\label{fig:B3}
\end{figure}
\end{exemple}

\subsection{The general case $Q^\vee\subseteq Y\subseteq P^\vee$}\label{subsec:general}
\hfill

We now have to see what happens if the $W$-lattice $Y$ is such that $Q^\vee \subsetneq Y \subsetneq P^\vee$. To simplify notations of this section, we identify a lattice $L\subset V^*$ with its translation group $\mathrm{t}(L)\subset \mathrm{Aff}(V^*)$. We shall give a fundamental domain for the \emph{intermediate affine Weyl group} $W_Y:=\mathrm{t}(Y)\rtimes W$ and prove that the quotient map is injective on it. It will be described as a \emph{polytopal complex}. In this paper, by a \emph{polytopal complex} we mean a finite union of closed convex polytopes, each one of which intersecting any other one along a common facet. For a more general definition and treatment of polytopal complexes, we refer the reader to \cite[Chapter 5, \S 5.1]{ziegler-poly}.

First, we shall identify $Y$ with a subgroup of $\Omega$. In fact, there is a correspondence between $W$-lattices $Q^\vee\subseteq\Lambda\subseteq P^\vee$ and the subgroups of $\Omega$. In order to state this correspondence properly, we temporarily drop the letter $Y$ and we work in the root system $\Phi$ only.

\begin{prop}\label{identlatticessubgroupsofOmega}
Recall that $\widehat{W_\mathrm{a}}\simeq W_\mathrm{a}\rtimes\Omega$ and denote by 
\[\pi : \widehat{W_\mathrm{a}}\twoheadrightarrow \Omega\]
the natural projection. Given a $W$-lattice $Q^\vee\subset\Lambda\subset P^\vee$, we define a subgroup $W_\Lambda:=\Lambda\rtimes W\le\widehat{W_\mathrm{a}}$. Then we have a bijective correspondence
\[\begin{array}{ccc}
\left\{\Lambda~;~Q^\vee\subseteq\Lambda\subseteq P^\vee~\text{is a } W\text{-lattice}\right\} & \stackrel{\tiny{\mathrm{1-1}}}\longleftrightarrow & \left\{H\le \Omega\right\} \\
\Lambda & \longmapsto & \Omega_\Lambda:=\pi(W_\Lambda) \\ \pi^{-1}(H)\cap P^\vee=:\Lambda(H) & \longmapsfrom & H \end{array}\]
Moreover, for a $W$-lattice $Q^\vee \subset \Lambda\subset P^\vee$, we have
\[[\Omega:\Omega_\Lambda]=[P^\vee:\Lambda],\]
or, equivalently
\[|\Omega_\Lambda|=[\Lambda:Q^\vee].\]
Finally, we have a decomposition
\[W_\Lambda\simeq W_\mathrm{a}\rtimes \Omega_\Lambda.\]
\end{prop}
\begin{proof}
First, we prove that the maps $\Omega_\bullet$ and $\Lambda(\bullet)$ are well-defined. It is clear that, given a $W$-lattice $\Lambda$, the set $\Omega_\Lambda=\pi(\Lambda\rtimes W)=\{\phi\in W_\Lambda~;~\phi(\mathcal{A}_0)=\mathcal{A}_0\}$ is a subgroup of $\Omega$. Conversely, le $H\le\Omega$ be a subgroup. We have that $\Lambda(H)=\pi^{-1}(H)\cap P^\vee$ is a subgroup of $\widehat{W_\mathrm{a}}$ and since $\Lambda(H)\subset P^\vee$, $\Lambda(H)$ is countable, hence discrete. We readily have $Q^\vee\subset\Lambda(H)\subset P^\vee$, which gives $\rk(\Lambda(H))=\rk(Q)=\rk(\Phi)$.
Moreover, if $x\in\Lambda(H)$ and $w\in W$, then we have $\pi(x)=\pi(x\cdot\mathrm{t}_0w)=\pi(\mathrm{t}_0w\cdot x)$ (recall that $\Omega$ is abelian) so $wx\in\pi^{-1}(H)$ that is, $wx\in\Lambda(H)$ and $\Lambda(H)$ is indeed a $W$-lattice lying between $Q^\vee$ and $P^\vee$.

We have to prove that $\Lambda\circ \Omega_\bullet=id$ and $\Omega_{\Lambda(\bullet)}=id$. Take a $W$-lattice $\Lambda$ between $Q^\vee$ and $P^\vee$ and let us show that $\Lambda(\Omega_\Lambda)=\Lambda$. If $\mu\in \Lambda$, then $\mu\in P^\vee$ and by construction, $\pi(\mathrm{t}_\mu)\in\pi(\Lambda\rtimes W)$ so $\mu\in \pi^{-1}(\pi(W_\Lambda))\cap P^\vee\stackrel{\tiny{\text{def}}}=\Lambda(\Omega_\Lambda)$. On the other hand, if $\mu\in \pi^{-1}(\pi(W_\Lambda))\cap P^\vee$, then there are $\nu\in \Lambda$ and $w\in W$ such that $\pi(\mathrm{t}_\mu)=\pi(\mathrm{t}_\nu w)$, that is, $1=\pi(\mathrm{t}_\nu w \mathrm{t}_{-\mu})=\pi(\mathrm{t}_{\nu-w(\mu)}w)$. Hence $\mathrm{t}_{\nu-w(\mu)}w\in W_\mathrm{a}$, which means that $\nu-w(\mu)\in Q^\vee\subset\Lambda$, but since $\nu\in\Lambda$ and $\Lambda$ is a $W$-lattice, we have $\mu\in\Lambda$.

Now, let $H\le\Omega$ be a subgroup and let us prove that $\Omega_{\Lambda(H)}=H$. Let $x\in \Omega_{\Lambda(H)}=\pi(W_{\Lambda(H)})$. There exist $\mathrm{t}_\mu w\in \Lambda(H)\rtimes W$ such that 
\[x=\pi(\mathrm{t}_\mu w)=\pi(\mathrm{t}_\mu 1\cdot \mathrm{t}_0w)=\pi(\mathrm{t}_\mu)\underbrace{\pi(\overset{\substack{W_\mathrm{a} \\ \inabove}}{\mathrm{t}_0w})}_{=1}=\pi(\mathrm{t}_\mu)\]
and this is in $H$ because $\mu\in\Lambda(H)=\pi^{-1}(H)\cap P^\vee$, so $x\in H$. Now, if $y\in H$, then we have $y=\pi(\mathrm{t}_\mu w)$ for some $\mu\in P^\vee$ and $w\in W$ because $\pi$ is onto, but the same calculation as above shows that $y=\pi(\mathrm{t}_\mu w)=\pi(\mathrm{t}_\mu)$, so $\mu\in P^\vee\cap\pi^{-1}(H)=\Lambda(H)$ and thus $y\in\pi(W_{\Lambda(H)})=\Omega_{\Lambda(H)}$, as required.

Next, the fact that $[P^\vee:\Lambda]\cdot [\Lambda:Q^\vee]=[P^\vee:Q^\vee]=|\Omega|$ ensures that the two equalities are indeed equivalent. But since $\ker(\pi)=W_\mathrm{a}\le W_\Lambda$ we have
\[\Omega_\Lambda=\pi(W_\Lambda)=\ima(W_\Lambda \stackrel{\tiny{\pi}}\to\Omega)\simeq W_\Lambda/\ker(\pi_{|W_\Lambda})=W_\Lambda/\ker(\pi)=W_\Lambda/W_\mathrm{a}\simeq \Lambda/Q^\vee.\]

Finally, using the above equality we see that the short exact sequence
\[\xymatrix{1 \ar[r] & W_\mathrm{a} \ar[r] & W_\Lambda \ar[r] & \Omega_\Lambda \ar[r] & 1}\]
is exact. It splits since the short exact sequence
\[\xymatrix{1 \ar[r] & W_\mathrm{a} \ar[r] & \widehat{W_\mathrm{a}} \ar^\pi[r] & \Omega \ar[r] & 1}\]
is split.
\end{proof}

We have a useful characterization of elements of $\Omega_\Lambda$, once $\Lambda$ is known.
\begin{lem}\label{characteltsofOmega_Lambda}
Let $\Lambda$ be a $W$-lattice between $Q^\vee$ and $P^\vee$ and consider $\omega_i,\omega_j\in\Omega$. Then, the following holds
\begin{enumerate}[label=\alph*)]
\item We have $\omega_i\in\Omega_\Lambda$ if and only if $\varpi_i^\vee\in\Lambda$.
\item The group $\Omega_\Lambda$ acts on the set ${\mathcal{A}_0}\cap\Lambda$.
\item We have $\omega_i\Omega_\Lambda=\omega_j\Omega_\Lambda$ if and only if $\varpi_i^\vee-\varpi_j^\vee\in\Lambda$.
\end{enumerate}
\end{lem}
\begin{proof}
\begin{enumerate}[label=\alph*)]
\item This is obvious, we have 
\[\omega_i\in\Omega_\Lambda~\Leftrightarrow~\omega_i\in W_\Lambda~\Leftrightarrow~\varpi_i^\vee=\omega_i(0)\in\Lambda.\]
\item If $\lambda\in{\mathcal{A}_0}\cap\Lambda$, then in particular we have $\lambda\in{\mathcal{A}_0}\cap P^\vee$ so by \cite[VI, \S 2.2, Proposition 5]{bourbaki456} there is some $k\in J$ such that $\lambda=\varpi_k^\vee\in\Lambda$. Thus, if $\omega_i\in\Omega_\Lambda$, we have $\omega_i(\lambda)=\omega_i(\varpi_k^\vee)=w_i(\varpi_k^\vee)+\varpi_i^\vee\in\Lambda$ because $\varpi_i^\vee\in\Lambda$ by $a)$ and $w_i(\varpi_k^\vee)\in\Lambda$ since $\Lambda$ is a $W$-lattice.
\item First, we compute
\[\omega_j^{-1}\omega_i=(\mathrm{t}_{\varpi_j^\vee}w_j)^{-1}\mathrm{t}_{\varpi_i^\vee}w_i=w_j^{-1}\mathrm{t}_{\varpi_i^\vee-\varpi_j^\vee}w_i=\mathrm{t}_{w_j^{-1}(\varpi_i^\vee-\varpi_j^\vee)}w_j^{-1}w_i.\]
Thus, we have
\[\omega_j^{-1}\omega_i\in\Omega_\Lambda~\Leftrightarrow~\mathrm{t}_{w_j^{-1}(\varpi_i^\vee-\varpi_j^\vee)}w_j^{-1}w_i\in\Omega_\Lambda~\Leftrightarrow~w_j^{-1}(\varpi_i^\vee-\varpi_j^\vee)\in\Lambda~\Leftrightarrow~\varpi_i^\vee-\varpi_j^\vee\in\Lambda.\]
\end{enumerate}
\end{proof}
Return to the case $\Lambda=Y$ and recall the notations from the preceding subsection: $I=\{1,\dotsc,r\}$, $\Pi=\{\alpha_i,~i\in I\}$, $\alpha_0=\sum_i n_i\alpha_i$ and $J=\{i\in I~;~n_i=1\}$. Recall also that the polytope
\[F:=F_{P^\vee}=\{\lambda\in{\mathcal{A}_0}~;~\forall j\in J,~\left<\lambda,\alpha_0+\alpha_i\right>\le1\}\]
given in Proposition \ref{funddomforOmega} is a fundamental domain for $\Omega$ in ${\mathcal{A}_0}$.

\begin{prop}\label{intersectionsalongfacets}
The following holds
\[\forall \omega\in\Omega\setminus\{1\},~\mathrm{codim}_F(F\cap\omega F)=1.\]
More precisely, for $1\ne \omega\in \Omega$, the polytopes $F$ and $\omega F$ intersect along their common facet $F\cap\{\lambda\in V^*~;~\left<\lambda,\alpha_0+\alpha_j\right>=1\}$, where $j\in J$ is such that $\omega(\varpi_j^\vee)=0$.
\end{prop}
\begin{proof}
We can write $\omega=\omega_i$ for some $i\in J$. Observe that if $\omega(\varpi_j^\vee)=0$ then, using the Lemmae \ref{actionwimonroots} and \ref{actionwionroots}, we compute for $\lambda \in V^*$,
\[\left<\omega^{-1}(\lambda),\alpha_0+\alpha_j\right>=\left<w_i^{-1}\lambda+\varpi_j^\vee,\alpha_0+\alpha_j\right>=2+\left<\lambda,w_i\alpha_0+w_i\alpha_j\right>=2-\left<\lambda,\alpha_0+\alpha_i\right>.\]
Hence, we have
\begin{equation}\label{exchangeroots}
\omega_i(\varpi_j^\vee)=0~\Rightarrow~\forall \lambda\in V^*,~\left<\lambda,\alpha_0+\alpha_j\right>+\left<\omega_i(\lambda),\alpha_0+\alpha_i\right>=2.
\end{equation}

With that being said, let us prove that 
\[F\cap \{\lambda~;~\left<\lambda,\alpha_0+\alpha_i\right>=1\}=F\cap\omega F.\]
Let $\lambda\in F$ such that $\left<\lambda,\alpha_0+\alpha_i\right>=1$. By (\ref{exchangeroots}), we have $\left<\omega^{-1}(\lambda),\alpha_0+\alpha_j\right>=1$ and if $k\ne j$, we have $w_i\alpha_k=\alpha_{k'}$ for some $k'\in J$ and we compute
\[\left<\omega^{-1}(\lambda),\alpha_0+\alpha_k\right>=\left<w_i^{-1}\lambda+\varpi_j^\vee,\alpha_0+\alpha_k\right>=1+\left<\lambda,\alpha_{k'}-\alpha_i\right>=\left<\lambda,\alpha_0+\alpha_{k'}\right>\le 1.\]
Hence $\omega^{-1}(\lambda)\in F$ and so $\lambda\in F\cap \omega F$.

Conversely, if $\lambda\in F\cap\omega F$, then by (\ref{exchangeroots}) we get
\[1\ge\left<\lambda,\alpha_0+\alpha_i\right>=2-\underbrace{\left<\omega^{-1}(\lambda),\alpha_0+\alpha_j\right>}_{\le1}\ge1\]
hence $\left<\lambda,\alpha_0+\alpha_i\right>=1$, as required.
\end{proof}

\begin{cor}\label{funddominA0intermediatelattice}
To the $W$-lattice $Y$ we associate the following subgroup $\Omega_Y$ of $\Omega$ from Proposition \ref{identlatticessubgroupsofOmega}:
\[\Omega_Y=\{\phi\in W_Y~;~\phi(\mathcal{A}_0)=\mathcal{A}_0\}\le\Omega.\]
Choose a representative $\omega\in\Omega$ for each coset $[\omega]\in\Omega/\Omega_Y$ and define
\[F_Y:=\bigcup_{[\omega]\in \Omega/\Omega_Y}\omega\cdot F_{P^\vee}.\]
Then $F_Y$ is a polytopal complex and is a fundamental domain for the action of $\Omega_Y$ on ${\mathcal{A}_0}$.
\end{cor}
\begin{proof}
By the Proposition \ref{intersectionsalongfacets}, for every subset $A\subseteq \Omega$, the subset $\bigcup_{a\in A}aF$ is a connected closed polytopal complex. 

The rest of this proof is very standard. Let $k:=[P^\vee:Y]=[\Omega:\Omega_Y]$, choose representatives $x_1,\dotsc,x_k\in \Omega$ such that $\Omega/\Omega_Y=\{x_i\Omega_Y,~1\le i \le k\}$ and define $F_Y:=\bigcup_{i=1}^k x_i F_{P^\vee}$. Let $\lambda\in{\mathcal{A}_0}$. Since $F_{P^\vee}$ is a fundamental domain for $\widehat{W_\mathrm{a}}$, there is some $\omega\in \Omega$ such that $\omega^{-1}(\lambda)\in F_{P^\vee}$ and writing $\omega=ux_i$ for some $1\le i \le k$ and some $u\in \Omega_Y$, we obtain $u^{-1}\lambda\in x_i F_{P^\vee}\subset F_Y$. On the other hand, if $u\in \Omega_Y$, then
\[F_Y\cap uF_Y=\bigcup_{1\le i,j\le k}x_iF_{P^\vee} \cap ux_jF_{P^\vee}=\bigcup_{i,j}x_i(F_{P^\vee}\cap x_i^{-1}ux_jF_{P^\vee}).\]
If $F_Y\cap uF_Y$ has non-empty interior, then there are some $i,j$ such that $F_{P^\vee}\cap ux_jx_i^{-1}F_{P^\vee}$ has non-empty interior. This implies that $ux_jx_i^{-1}=1$, thus $x_ix_j^{-1}\in \Omega_Y$ so $i=j$ and $u=1$.
\end{proof}

From this, we deduce the main result of this section:
\begin{theo}\label{funddomintermediatelattice}
If $K$ is only supposed to be compact, let $W_Y:=\mathrm{t}(Y)\rtimes W$ (where $Y=Y(T)$ is the cocharacter lattice) and consider the subgroup
\[\Omega_Y=\{\phi\in W_Y~;~\phi(\mathcal{A}_0)=\mathcal{A}_0\}\le\Omega\]
and choose a representative for each coset in $\Omega/\Omega_Y$. Define
\[F_Y:=\bigcup_{[\omega]\in \Omega/\Omega_Y}\omega\cdot F_{P^\vee},\]
where $F_{P^\vee}$ is the fundamental domain for $\Omega$ in ${\mathcal{A}_0}$ given in Proposition \ref{funddomforOmega}. Then $F_Y$ is a polytopal complex and is a fundamental domain for the action of $W_Y\simeq W_\mathrm{a}\rtimes \Omega_Y$ on $V^*$.
\end{theo}

We still have to investigate the question of the injectivity of the exponential map on $F_Y$. This is done in the following proposition:
\begin{prop}\label{injectivityofexponfunddomintermediatecase}
The hypothesis of the Lemma \ref{fromlatticetotorus} is fulfilled under the ones of the Theorem \ref{funddomintermediatelattice}. In other words, one has
\[F_Y\cap(Y\setminus\{0\}+F_Y)=\emptyset.\]
\end{prop}
\begin{proof}
Let $\lambda,\lambda'\in F_Y$ and suppose by contradiction that $0\ne \mu:=\lambda-\lambda'\in Y$. In particular we have $\mu\in P^\vee$ and since $\lambda,\lambda'\in{\mathcal{A}_0}$, we have $\left<\mu,\alpha\right>\in\{-1,0,1\}$ for every $\alpha\in\Phi$. Without loss of generality, as $\Omega_Y$ acts on $Y$, we may assume that the trivial coset in $\Omega/\Omega_Y$ is represented by $1\in\Omega$. 

First, suppose that $\mu\notin{\mathcal{A}_0}$. Since $\left<\mu,\alpha_0\right>\le1$, we may choose $i\in I$ such that $\left<\mu,\alpha_i\right>=-1$. Then, we have $1\ge\left<\lambda',\alpha_i\right>=\left<\lambda,\alpha_i\right>+1\ge1$ so $\left<\lambda',\alpha_i\right>=1$ and thus $\left<\lambda',\alpha_k\right>=0$ for $k\ne i$ since $\lambda'\in{\mathcal{A}_0}$ and therefore $\lambda'=\varpi_i^\vee$. 

Now, if there is some $j$ such that $\left<\mu,\alpha_j\right>\ne0$, then $\left<\mu,\alpha_j\right>=1$ and similarly as above, we get $\left<\lambda,\alpha_j\right>=1$ and $\lambda=\varpi_j^\vee$. Notice that $j\in J$ because $\lambda\in {\mathcal{A}_0}$.

\underline{Claim} : For every $h\in J$ and every $\omega\in\Omega$, we have $\omega=\omega_h$ as soon as $\varpi_h^\vee\in \omega F_{P^\vee}$.
\newline Indeed, it suffices to show that $\omega^{-1}(\varpi_h^\vee)=0$. Suppose otherwise, then there exists $\ell\in J$ such that $\omega^{-1}(\varpi_h^\vee)=\varpi_\ell^\vee$ and this would be in $F_{P^\vee}$, a contradiction since $\left<\varpi_\ell^\vee,\alpha_0+\alpha_\ell\right>=2$ and $\ell\in J$.

This proves that if $\lambda=\varpi_j^\vee\in \omega F_{P^\vee}$ for some $\omega$, then $\omega=\omega_j$ and similarly for $\lambda'$. Since $F_Y$ is the union of the $\omega F_{P^\vee}$ for $\omega$ describing a set of coset representatives modulo $\Omega_Y$, and since $\lambda-\lambda'=\varpi_j^\vee-\varpi_i^\vee\in Y$, we must have $\omega_i\Omega_Y=\omega_j\Omega_Y$ by the Lemma \ref{characteltsofOmega_Lambda} and thus $\lambda=\lambda'$, a contradiction.

Hence, if $\mu\notin{\mathcal{A}_0}$, then for all $k\ne i$, we have $\left<\mu,\alpha_k\right>=0$, so $\mu=-\varpi_i^\vee\in Y$ and $\lambda'=\lambda+\varpi_i^\vee$. Henceforth, $\lambda'=\varpi_i^\vee\in Y$ and the only $\omega\in\Omega$ such that $\varpi_i^\vee\in\omega F_{P^\vee}$ is $\omega=\omega_i$ by the claim and this is in $\Omega_Y$ by the Lemma \ref{characteltsofOmega_Lambda}. By definition of $F_Y$ and since the trivial coset in $\Omega/\Omega_Y$ is represented by $1$, this implies that $\varpi_i^\vee\in Y\cap F_{P^\vee}\subset P^\vee\cap F_{P^\vee}=\{0\}$ by the Lemma \ref{A0inter(A0+P)=0}, a contradiction again.

The only remaining possibility is $\mu\in{\mathcal{A}_0}$. But in this case $\mu\in Y\cap{\mathcal{A}_0}$ and $\mu=\varpi_i^\vee\in Y$ for some $i\in J$. Therefore, we obtain as above $\left<\lambda,\alpha_i\right>=\left<\lambda',\alpha_i\right>+1$, so $\left<\lambda,\alpha_i\right>=1$ and so $\lambda=\varpi_i^\vee\in Y$, but we have seen this to be impossible.
\end{proof}

\begin{cor}\label{funddomforKanycase}
If $K$ is a compact connected Lie group and if $Y:=X(T)^\vee$ is the cocharacter lattice of the chosen maximal torus $T\le K$, let $F_Y$ be the fundamental domain from the Theorem \ref{funddomintermediatelattice}. Then $\exp(F_Y)$ is homeomorphic to the polytopal complex $F_Y$ and is a fundamental domain for the action of $W$ on $T$. 
\end{cor}

\begin{prop}\label{anotherfunddomintermediatecases}
Let
\[J_Y:=\{j\in J~;~\varpi_j^\vee\in Y\}=\{j\in J~;~\omega_j\in\Omega_Y\}\] and define
\[F'_Y:=\{\lambda\in{\mathcal{A}_0}~;~\forall j\in J_Y,~\left<\lambda,\alpha_0+\alpha_j\right>\le1\}.\]
Then $F_Y'$ is a convex polytope and is a fundamental domain for $W_Y=\mathrm{t}(Y)\rtimes W$ on $V^*$.
\end{prop}
\begin{proof}
We may proceed as in the adjoint case, replacing $\Omega$ by $\Omega_Y$. It clearly suffices to show that $F_Y'$ is a fundamental domain for $\Omega_Y$ in ${\mathcal{A}_0}$.

First, since $F_Y'$ is an intersection of closed half-spaces, it is a convex polytope in ${\mathcal{A}_0}$. In particular, it is closed and connected.

If $\omega_i\in\Omega_Y$ and $\lambda\in F_Y'$, then there exists $j\in J$ such that $\omega_i(\varpi_j^\vee)=0$ and we have $j\in J_Y$ since $\Omega_Y$ acts on $Y\cap{\mathcal{A}_0}$. We have
\[\left<\omega_i(\lambda),\alpha_0+\alpha_i\right>=2-\left<\lambda,\alpha_0+\alpha_j\right>\]
and
\[F_Y'\cap \omega_i^{-1}F'_Y\subset {\mathcal{A}_0}\cap \{\mu\in V^*~;~\left<\mu,\alpha_0+\alpha_i\right>=1\}\]
and thus $F_Y'\cap \omega_i^{-1}F_Y'$ has empty interior in ${\mathcal{A}_0}$

Let now $\lambda\in{\mathcal{A}_0}$ and choose $i\in J_Y$ such that $\left<\lambda,\alpha_i\right>=\max_{j\in J_Y}\left<\lambda,\alpha_j\right>$. If $\left<\lambda,\alpha_0+\alpha_i\right>\le1$, then $\lambda\in F_Y'$. Otherwise, choose $j\in J_Y$ such that $\omega_i(\varpi_j^\vee)\ne0$ and take some $k\in J$ such that $\omega_i(\varpi_j^\vee)=\varpi_k^\vee$. Then $k\in J_Y$ because $j\in J_Y$ and $\omega_i\in\Omega_Y$ by the Lemma \ref{characteltsofOmega_Lambda}. Hence
\[\left<\omega_i^{-1}(\lambda),\alpha_0+\alpha_j\right>=1+\underbrace{\left<\lambda,\alpha_k-\alpha_i\right>}_{\le0}\le1.\]
Now, if $j\in J_Y$ is such that $\omega_i(\varpi_j^\vee)=0$, then $w_i(\alpha_j)=-\alpha_0$ and
\[\left<\omega_i^{-1}(\lambda),\alpha_0+\alpha_j\right>=2-\underbrace{\left<\lambda,\alpha_0+\alpha_i\right>}_{>1}<1\]
and thus $\omega_i^{-1}(\lambda)\in F_Y'$, as required.
\end{proof}

Reproducing the proof of the Proposition \ref{verticesofFP} \emph{verbatim} in this case yields the following description of the vertices of $F_Y'$:
\begin{prop}\label{verticesofFY}
Define the following set of isobarycenters
\[\mathcal{B}_m^Y:=\left\{\frac{1}{|J'_Y|+1}\sum_{j\in J'_Y}\varpi_j^\vee~;~\emptyset\subseteq J'_Y\subseteq J_Y\right\},\]
where $J_Y=\{j\in J~;~\varpi_j^\vee\in Y\}$. Then the vertices of the polytope $F_Y'$ are given by
\[\vertices(F_Y')=\mathcal{B}_m^Y\cup\left\{\frac{\varpi_i^\vee}{n_i}\right\}_{i\in I\setminus J_Y}.\]
\end{prop}

\begin{exemple}\label{F'notinj}
This fundamental domain $F_Y'$ can contain different points that are congruent modulo $Y$.

For instance, let $(X(T),\Phi,Y(T),\Phi^\vee)$ be the root datum of the Lie group $SU_4(\C)/\{\pm1\}$, with maximal torus being the the image in the quotient of the diagonal matrices of $SU_4$. In particular, $\Phi$ is of type $A_3$. Denote the simple roots by $\Pi:=\{\alpha,\beta,\gamma\}$ with extended Dynkin diagram

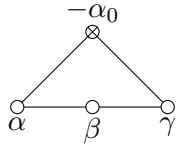
\begin{center}
\begin{tikzpicture}
	\coordinate (a) at (-1,0);
	\coordinate (b) at (0,0);
	\coordinate (c) at (1,0);
	\coordinate (a0) at (0,1);
	
	\draw (a) node[below]{$\alpha$};
	\draw (b) node[below]{$\beta$};
	\draw (c) node[below]{$\gamma$};
	\draw (a0) node[above]{$-\alpha_0$};
	
	\draw (a)--(b)
	(b)--(c)
	(c)--(a0)
	(a0)--(a);
	
	\fill[fill=white] (a) circle (2.5pt);
	\fill[fill=white] (b) circle (2.5pt);
	\fill[fill=white] (c) circle (2.5pt);
	\fill[fill=white] (a0) circle (2.5pt);
	\node[mark size=2.5pt] at (a0) {\pgfuseplotmark{otimes}};
	\draw (a) circle (2.5pt);
	\draw (b) circle (2.5pt);
	\draw (c) circle (2.5pt);
\end{tikzpicture}	
\end{center}

Let also $\varpi_\alpha^\vee$, $\varpi_\beta^\vee$ and $\varpi_\gamma^\vee$ be the corresponding fundamental weights. We have (see \cite[\S 13.2, Table 1]{humlie})
\[\varpi_\alpha=\frac{3\alpha+2\beta+\gamma}{4},~\varpi_\beta=\frac{2\alpha+4\beta+2\gamma}{4},~\varpi_\gamma=\frac{\alpha+2\beta+3\gamma}{4},\]
and every fundamental weight is minuscule since the highest root is $\alpha_0=\alpha+\beta+\gamma$. In this case, the cocharacter lattice $Y=Y(T)$ is given by
\[Y=\Z\left<\alpha^\vee\right>\oplus\Z\left<\beta^\vee\right>\oplus\Z\left<\frac{\alpha^\vee+\gamma^\vee}{2}\right>=:\Z\left<\alpha^\vee\right>\oplus\Z\left<\beta^\vee\right>\oplus\Z\left<y\right>.\]

We have
\[\gamma^\vee=2y-\alpha^\vee\in Y,~\varpi_\alpha^\vee=\frac{\alpha^\vee+\beta^\vee+y}{2}\notin Y,~\varpi_\gamma^\vee=\frac{\beta^\vee-\alpha^\vee+3y}{2}\notin Y\]
and hence $Q^\vee\subsetneq Y\subsetneq P^\vee$. Furthermore, $Y\cap{\mathcal{A}_0}=\{\varpi_\beta^\vee\}$ and
\[F_Y'=\{\lambda\in{\mathcal{A}_0}~;~\left<\lambda,\alpha_0+\beta\right>\le1\}=\{\lambda\in \conv(0,\varpi_\alpha^\vee,\varpi_\beta^\vee,\varpi_\gamma^\vee)~;~\left<\lambda,\alpha+2\beta+\gamma\right>\le1\}.\]
So $\varpi_\alpha^\vee$ and $\varpi_\gamma^\vee$ are distinct elements of $F_Y'$ and however,
\[\varpi_\gamma^\vee-\varpi_\alpha^\vee=\varpi_\beta^\vee-\alpha^\vee-\beta^\vee\in Y.\]
Note that, in this case we have
\[\Omega=\left<\omega_\alpha\right>=\{1,\underbrace{\mathrm{t}_{\varpi_\alpha^\vee}(s_\alpha s_\beta s_\gamma)}_{\omega_\alpha},\underbrace{\mathrm{t}_{\varpi_\beta^\vee}(s_\beta s_\gamma s_\alpha s_\beta)}_{\omega_\beta},\underbrace{\mathrm{t}_{\varpi_\gamma^\vee}(s_\gamma s_\beta s_\alpha)}_{\omega_\gamma}\}\simeq \Z/4\Z\]
and 
\[\Omega_Y=\left<\omega_\beta\right>=\{1,\omega_\beta\}\le\Omega~~\text{and}~~\Omega/\Omega_Y\simeq\pi_1(SU_4(\C)/\{\pm1\})\simeq\Z/2\Z.\]
By the Proposition \ref{identlatticessubgroupsofOmega}, we see that $Y$ is the only $W$-lattice that stands strictly between $Q^\vee$ and $P^\vee$. Thus, the group $SU_4(\C)/\{\pm1\}$ is the only non-adjoint and non-simply-connected compact Lie group of type $A_3$.

Besides, observe that we have indeed
\[F_{P^\vee}=\conv\left(0,\frac{\varpi_\alpha}{2},\frac{\varpi_\beta}{2},\frac{\varpi_\gamma}{2},\frac{\varpi_\alpha+\varpi_\beta}{3},\frac{\varpi_\beta+\varpi_\gamma}{3},\frac{\varpi_\alpha+\varpi_\gamma}{3},\frac{\varpi_\alpha+\varpi_\beta+\varpi_\gamma}{4}\right).\]

\begin{figure}[h!]
\begin{tikzpicture}[x={(-1cm,0cm)},y={(0cm,-1cm)},z={(-5mm,2mm)},scale=2.25]
  \coordinate (z) at (0,0,0);

  \coordinate (a) at (1,-1,0);
  \coordinate (b) at (0,1,-1);
  \coordinate (c) at (-1,-1,0);
  \coordinate (ab) at (1,0,-1);
  \coordinate (bc) at (-1,0,-1);
  \coordinate (abc) at (0,-1,-1);
  \coordinate (ma) at (-1,1,0);
  \coordinate (mb) at (0,-1,1);
  \coordinate (mc) at (1,1,0);
  \coordinate (mab) at (-1,0,1);
  \coordinate (mbc) at (1,0,1);
  \coordinate (mabc) at (0,1,1);
  
  \coordinate (la) at (1/2,-1/2,-1/2);
  \coordinate (lb) at (0,0,-1);
  \coordinate (lc) at (-1/2,-1/2,-1/2);
  
  \coordinate (lad) at (1/4,-1/4,-1/4);
  \coordinate (lbd) at (0,0,-1/2);
  \coordinate (lcd) at (-1/4,-1/4,-1/4);
  \coordinate (lbct) at (-1/6,-1/6,-1/2);
  \coordinate (lact) at (0,-1/3,-1/3);
  \coordinate (labt) at (1/6,-1/6,-1/2);
  \coordinate (labcq) at (0,-1/4,-1/2);
  
  \coordinate (barlac) at (0,-1/2,-1/2);
  \coordinate (barlbc) at (-1/4,-1/4,-3/4);
  \coordinate (barlabc) at (0,-1/3,-2/3);
  
  \fill[fill=black] (z) circle (1pt);
  \fill[fill=black] (a) circle (1pt);
  \fill[fill=black] (b) circle (1pt);
  \fill[fill=black] (c) circle (1pt);
  \fill[fill=black] (ab) circle (1pt);
  \fill[fill=black] (bc) circle (1pt);
  \fill[fill=black] (abc) circle (1pt);
  \fill[fill=black] (ma) circle (1pt);
  \fill[fill=black] (mb) circle (1pt);
  \fill[fill=black] (mc) circle (1pt);
  \fill[fill=black] (mab) circle (1pt);
  \fill[fill=black] (mbc) circle (1pt);
  \fill[fill=black] (mabc) circle (1pt);
  
  \draw (a) node[above]{$\alpha$};
  \draw (b) node[below]{$\beta$};
  \draw (c) node[above]{$\gamma$};
  \draw (ab) node[left]{$\alpha+\beta$};
  \draw (bc) node[right]{$\beta+\gamma$};
  \draw (abc) node[above]{$\alpha_0$};
  \draw (la) node[above left]{$\varpi_\alpha$};
  \draw (lb) node[below]{$\varpi_\beta$};
  \draw (lc) node[right]{$\varpi_\gamma$};
  
  \draw (z)--(a);
  \draw (z)--(b);
  \draw (z)--(c);
  \draw (z)--(ab);
  \draw (z)--(bc);
  \draw (z)--(abc);
  \draw[dashed] (z)--(ma);
  \draw[dashed] (z)--(mb);
  \draw[dashed] (z)--(mc);
  \draw[dashed] (z)--(mab);
  \draw[dashed] (z)--(mbc);
  \draw[dashed] (z)--(mabc);
  
  \fill[fill=red] (la) circle (1pt);
  \fill[fill=red] (lb) circle (1pt);
  \fill[fill=red] (lc) circle (1pt);

  \fill[fill=blue,opacity=0.5] (z)--(la)--(lb)--(z);
  \fill[fill=blue,opacity=0.5] (z)--(lb)--(lc)--(z);
  \fill[fill=blue,opacity=0.5] (z)--(lc)--(la)--(z);
  \fill[fill=blue,opacity=0.5] (la)--(lb)--(lc)--(la);

  \draw (z)--(lc);
  \draw (z)--(lb);
  \draw (z)--(la);
  \draw (la)--(lb);
  \draw (lb)--(lc);
  \draw (lc)--(la);
  
  \coordinate (x) at (1.3,-1.3,-1.3);
  \coordinate (y) at (0,0,-2.5);
  \coordinate (t) at (-1.3,-1.3,-1.3);
  
  \fill[fill=gray,opacity=0.2] (z)--(x)--(y);
  \fill[fill=gray,opacity=0.2] (z)--(y)--(t);
  \fill[fill=gray,opacity=0.2] (z)--(t)--(x);
  \draw[dotted] (z)--(x);
  \draw[dotted] (z)--(y);
  \draw[dotted] (z)--(t);
\end{tikzpicture}
\caption{The fundamental alcove ${\mathcal{A}_0}=\conv(0,\varpi_\alpha,\varpi_\beta,\varpi_\gamma)$ in type $A_3$.}\label{fig:A3}
\end{figure}
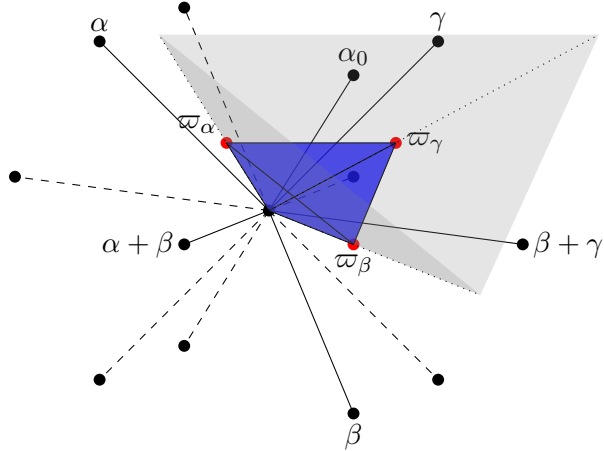

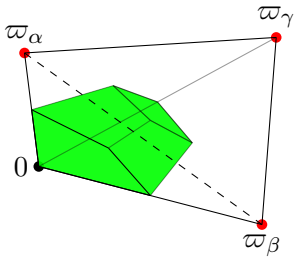
\begin{figure}[h!]
\begin{tikzpicture}[x={(-1cm,0cm)},y={(0cm,-1cm)},z={(-5mm,2mm)},scale=4,rotate around y=15]
  \coordinate (z) at (0,0,0);

  \coordinate (a) at (1,-1,0);
  \coordinate (b) at (0,1,-1);
  \coordinate (c) at (-1,-1,0);
  \coordinate (ab) at (1,0,-1);
  \coordinate (bc) at (-1,0,-1);
  \coordinate (abc) at (0,-1,-1);
  \coordinate (ma) at (-1,1,0);
  \coordinate (mb) at (0,-1,1);
  \coordinate (mc) at (1,1,0);
  \coordinate (mab) at (-1,0,1);
  \coordinate (mbc) at (1,0,1);
  \coordinate (mabc) at (0,1,1);
  
  \coordinate (la) at (1/2,-1/2,-1/2);
  \coordinate (lb) at (0,0,-1);
  \coordinate (lc) at (-1/2,-1/2,-1/2);
  
  \coordinate (lad) at (1/4,-1/4,-1/4);
  \coordinate (lbd) at (0,0,-1/2);
  \coordinate (lcd) at (-1/4,-1/4,-1/4);
  \coordinate (lbct) at (-1/6,-1/6,-1/2);
  \coordinate (lact) at (0,-1/3,-1/3);
  \coordinate (labt) at (1/6,-1/6,-1/2);
  \coordinate (labcq) at (0,-1/4,-1/2);
  
  \coordinate (barlac) at (0,-1/2,-1/2);
  \coordinate (barlbc) at (-1/4,-1/4,-3/4);
  \coordinate (barlabc) at (0,-1/3,-2/3);
  
  \fill[fill=black] (z) circle (0.5pt);
  
  \draw (la) node[above]{$\varpi_\alpha$};
  \draw (lb) node[below]{$\varpi_\beta$};
  \draw (lc) node[above]{$\varpi_\gamma$};
  \draw (z) node[left]{$0$};
  
  \fill[fill=red] (la) circle (0.5pt);
  \fill[fill=red] (lb) circle (0.5pt);
  \fill[fill=red] (lc) circle (0.5pt);
  
  \fill[fill=green,opacity=0.7] (z)--(lad)--(labt)--(lbd)--(z);
  \fill[fill=green,opacity=0.7] (z)--(lbd)--(lbct)--(lcd)--(z);
  \fill[fill=green,opacity=0.7] (z)--(lcd)--(lact)--(lad)--(z);
  \fill[fill=green,opacity=0.7] (lad)--(labt)--(labcq)--(lact)--(lad);
  \fill[fill=green,opacity=0.7] (lbd)--(lbct)--(labcq)--(labt)--(lbd);
  \fill[fill=green,opacity=0.7] (lcd)--(lact)--(labcq)--(lbct)--(lcd);
  \draw (lad)--(labt);
  \draw (lbd)--(labt);
  \draw[opacity=0.4] (lad)--(lact);
  \draw[opacity=0.4] (lcd)--(lact);
  \draw[opacity=0.4] (lcd)--(lbct);
  \draw[opacity=0.4] (lbd)--(lbct);
  \draw[opacity=0.4] (labt)--(labcq);
  \draw[opacity=0.4] (lbct)--(labcq);
  \draw[opacity=0.4] (lact)--(labcq);
  \draw (z)--(la);
  \draw (z)--(lb);
  \draw[opacity=0.4] (z)--(lc);
  \draw[dashed] (la)--(lb);
  \draw (lb)--(lc);
  \draw (lc)--(la);
  \draw (z)--(lad);
  \draw (z)--(lbd);
  \draw[opacity=0.4] (lact)--(labcq);
  \draw[opacity=0.4] (lbct)--(labcq);
  \draw[opacity=0.4] (lcd)--(lbct);
  \draw[opacity=0.4] (lcd)--(lact);
  
\end{tikzpicture}
\caption{The fundamental domain $F_{P^\vee}$ for $\Omega\simeq \Z/4\Z$ in ${\mathcal{A}_0}$.}
\label{otherfunddom}
\end{figure}

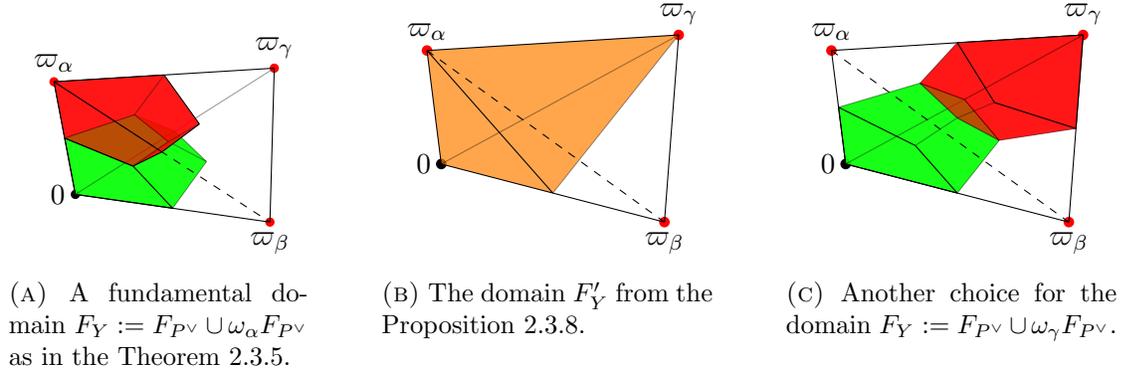
\begin{figure}[h!]
\subcaptionbox{A fundamental domain $F_Y:=F_{P^\vee}\cup \omega_\alpha F_{P^\vee}$ as in the Theorem \ref{funddomintermediatelattice}.}
{
\begin{tikzpicture}[x={(-1cm,0cm)},y={(0cm,-1cm)},z={(-5mm,2mm)},scale=3.5,rotate around y=15,rotate around x=-5]
  \coordinate (z) at (0,0,0);

  \coordinate (a) at (1,-1,0);
  \coordinate (b) at (0,1,-1);
  \coordinate (c) at (-1,-1,0);
  \coordinate (ab) at (1,0,-1);
  \coordinate (bc) at (-1,0,-1);
  \coordinate (abc) at (0,-1,-1);
  \coordinate (ma) at (-1,1,0);
  \coordinate (mb) at (0,-1,1);
  \coordinate (mc) at (1,1,0);
  \coordinate (mab) at (-1,0,1);
  \coordinate (mbc) at (1,0,1);
  \coordinate (mabc) at (0,1,1);
  
  \coordinate (la) at (1/2,-1/2,-1/2);
  \coordinate (lb) at (0,0,-1);
  \coordinate (lc) at (-1/2,-1/2,-1/2);
  
  \coordinate (lad) at (1/4,-1/4,-1/4);
  \coordinate (lbd) at (0,0,-1/2);
  \coordinate (lcd) at (-1/4,-1/4,-1/4);
  \coordinate (lbct) at (-1/6,-1/6,-1/2);
  \coordinate (lact) at (0,-1/3,-1/3);
  \coordinate (labt) at (1/6,-1/6,-1/2);
  \coordinate (labcq) at (0,-1/4,-1/2);
  
  \coordinate (barlab) at (1/4,-1/4,-3/4);
  \coordinate (barlac) at (0,-1/2,-1/2);
  \coordinate (barlbc) at (-1/4,-1/4,-3/4);
  \coordinate (barlabc) at (0,-1/3,-2/3);
  
  \fill[fill=black] (z) circle (0.5pt);
  
  \draw (la) node[above]{$\varpi_\alpha$};
  \draw (lb) node[below]{$\varpi_\beta$};
  \draw (lc) node[above]{$\varpi_\gamma$};
  \draw (z) node[left]{$0$};
  
  \fill[fill=red] (la) circle (0.5pt);
  \fill[fill=red] (lb) circle (0.5pt);
  \fill[fill=red] (lc) circle (0.5pt);
  
  \fill[fill=green,opacity=0.7] (z)--(lad)--(labt)--(lbd)--(z);
  \fill[fill=green,opacity=0.7] (z)--(lbd)--(lbct)--(lcd)--(z);
  \fill[fill=green,opacity=0.7] (z)--(lcd)--(lact)--(lad)--(z);
  \fill[fill=green,opacity=0.7] (lad)--(labt)--(labcq)--(lact)--(lad);
  \fill[fill=green,opacity=0.7] (lbd)--(lbct)--(labcq)--(labt)--(lbd);
  \fill[fill=green,opacity=0.7] (lcd)--(lact)--(labcq)--(lbct)--(lcd);
  \draw (lad)--(labt);
  \draw (lbd)--(labt);
  \draw[opacity=0.4] (lad)--(lact);
  \draw[opacity=0.4] (lcd)--(lact);
  \draw[opacity=0.4] (lcd)--(lbct);
  \draw[opacity=0.4] (lbd)--(lbct);
  \draw[opacity=0.4] (labt)--(labcq);
  \draw[opacity=0.4] (lbct)--(labcq);
  \draw[opacity=0.4] (lact)--(labcq);
  \draw (z)--(la);
  \draw (z)--(lb);
  \draw[opacity=0.4] (z)--(lc);
  \draw[dashed] (la)--(lb);
  \draw (lb)--(lc);
  \draw (lc)--(la);
  \draw (z)--(lad);
  \draw (z)--(lbd);
  
  \fill[fill=red,opacity=0.7] (la)--(lad)--(lact)--(barlac)--(la);
  \fill[fill=red,opacity=0.7] (la)--(lad)--(labt)--(barlab)--(la);
  \fill[fill=red,opacity=0.7] (la)--(barlac)--(barlabc)--(barlab)--(la);
  \fill[fill=red,opacity=0.7] (barlac)--(lact)--(labcq)--(barlabc)--(barlac);
  \fill[fill=red,opacity=0.7] (barlab)--(labt)--(labcq)--(barlabc)--(barlab);
  \draw (la)--(lad);
  \draw (la)--(barlab);
  \draw (la)--(barlac);
  \draw[opacity=0.4] (lact)--(barlac);
  \draw (labt)--(barlab);
  \draw[opacity=0.4] (lact)--(labcq);
  \draw[opacity=0.4] (labt)--(labcq);
  \draw (barlab)--(barlabc);
  \draw (barlac)--(barlabc);
  \draw[opacity=0.4] (barlabc)--(labcq);
  \draw (lad)--(labt);
  \draw[opacity=0.4] (lad)--(lact);  
\end{tikzpicture}
}
\hspace{7mm}
\subcaptionbox{The domain $F_Y'$ from the Proposition \ref{anotherfunddomintermediatecases}.}
{
\begin{tikzpicture}[x={(-1cm,0cm)},y={(0cm,-1cm)},z={(-5mm,2mm)},scale=4,rotate around y=15]
  \coordinate (z) at (0,0,0);

  \coordinate (a) at (1,-1,0);
  \coordinate (b) at (0,1,-1);
  \coordinate (c) at (-1,-1,0);
  \coordinate (ab) at (1,0,-1);
  \coordinate (bc) at (-1,0,-1);
  \coordinate (abc) at (0,-1,-1);
  \coordinate (ma) at (-1,1,0);
  \coordinate (mb) at (0,-1,1);
  \coordinate (mc) at (1,1,0);
  \coordinate (mab) at (-1,0,1);
  \coordinate (mbc) at (1,0,1);
  \coordinate (mabc) at (0,1,1);
  
  \coordinate (la) at (1/2,-1/2,-1/2);
  \coordinate (lb) at (0,0,-1);
  \coordinate (lc) at (-1/2,-1/2,-1/2);
  
  \coordinate (lad) at (1/4,-1/4,-1/4);
  \coordinate (lbd) at (0,0,-1/2);
  \coordinate (lcd) at (-1/4,-1/4,-1/4);
  \coordinate (lbct) at (-1/6,-1/6,-1/2);
  \coordinate (lact) at (0,-1/3,-1/3);
  \coordinate (labt) at (1/6,-1/6,-1/2);
  \coordinate (labcq) at (0,-1/4,-1/2);
  
  \coordinate (barlac) at (0,-1/2,-1/2);
  \coordinate (barlbc) at (-1/4,-1/4,-3/4);
  \coordinate (barlabc) at (0,-1/3,-2/3);
  
  \fill[fill=black] (z) circle (0.5pt);
  
  \draw (la) node[above]{$\varpi_\alpha$};
  \draw (lb) node[below]{$\varpi_\beta$};
  \draw (lc) node[above]{$\varpi_\gamma$};
  \draw (z) node[left]{$0$};
  
  \fill[fill=red] (la) circle (0.5pt);
  \fill[fill=red] (lb) circle (0.5pt);
  \fill[fill=red] (lc) circle (0.5pt);
  
  \fill[fill=orange,opacity=0.7] (z)--(lbd)--(lc)--(la)--(z);
  \draw (la)--(lbd);
  \draw[opacity=0.4] (lc)--(lbd);
  \draw (la)--(lc);
  \draw (z)--(lb);
  \draw (z)--(la);
  \draw[opacity=0.4] (z)--(lc);
  \draw[dashed] (la)--(lb);
  \draw (lb)--(lc);
\end{tikzpicture}
}
\hspace{7mm}
\subcaptionbox{Another choice for the domain $F_Y:=F_{P^\vee}\cup \omega_\gamma F_{P^\vee}$.}
{
\begin{tikzpicture}[x={(-1cm,0cm)},y={(0cm,-1cm)},z={(-5mm,2mm)},scale=4,rotate around y=15]
  \coordinate (z) at (0,0,0);

  \coordinate (a) at (1,-1,0);
  \coordinate (b) at (0,1,-1);
  \coordinate (c) at (-1,-1,0);
  \coordinate (ab) at (1,0,-1);
  \coordinate (bc) at (-1,0,-1);
  \coordinate (abc) at (0,-1,-1);
  \coordinate (ma) at (-1,1,0);
  \coordinate (mb) at (0,-1,1);
  \coordinate (mc) at (1,1,0);
  \coordinate (mab) at (-1,0,1);
  \coordinate (mbc) at (1,0,1);
  \coordinate (mabc) at (0,1,1);
  
  \coordinate (la) at (1/2,-1/2,-1/2);
  \coordinate (lb) at (0,0,-1);
  \coordinate (lc) at (-1/2,-1/2,-1/2);
  
  \coordinate (lad) at (1/4,-1/4,-1/4);
  \coordinate (lbd) at (0,0,-1/2);
  \coordinate (lcd) at (-1/4,-1/4,-1/4);
  \coordinate (lbct) at (-1/6,-1/6,-1/2);
  \coordinate (lact) at (0,-1/3,-1/3);
  \coordinate (labt) at (1/6,-1/6,-1/2);
  \coordinate (labcq) at (0,-1/4,-1/2);
  
  \coordinate (barlac) at (0,-1/2,-1/2);
  \coordinate (barlbc) at (-1/4,-1/4,-3/4);
  \coordinate (barlabc) at (0,-1/3,-2/3);
  
  \fill[fill=black] (z) circle (0.5pt);
  
  \draw (la) node[above]{$\varpi_\alpha$};
  \draw (lb) node[below]{$\varpi_\beta$};
  \draw (lc) node[above]{$\varpi_\gamma$};
  \draw (z) node[left]{$0$};
  
  \fill[fill=red] (la) circle (0.5pt);
  \fill[fill=red] (lb) circle (0.5pt);
  \fill[fill=red] (lc) circle (0.5pt);
  
  \fill[fill=green,opacity=0.7] (z)--(lad)--(labt)--(lbd)--(z);
  \fill[fill=green,opacity=0.7] (z)--(lbd)--(lbct)--(lcd)--(z);
  \fill[fill=green,opacity=0.7] (z)--(lcd)--(lact)--(lad)--(z);
  \fill[fill=green,opacity=0.7] (lad)--(labt)--(labcq)--(lact)--(lad);
  \fill[fill=green,opacity=0.7] (lbd)--(lbct)--(labcq)--(labt)--(lbd);
  \fill[fill=green,opacity=0.7] (lcd)--(lact)--(labcq)--(lbct)--(lcd);
  \draw (lad)--(labt);
  \draw (lbd)--(labt);
  \draw[opacity=0.4] (lad)--(lact);
  \draw (lcd)--(lact);
  \draw (lcd)--(lbct);
  \draw[opacity=0.4] (lbd)--(lbct);
  \draw[opacity=0.4] (labt)--(labcq);
  \draw[opacity=0.4] (lbct)--(labcq);
  \draw[opacity=0.4] (lact)--(labcq);
  \draw (z)--(la);
  \draw (z)--(lb);
  \draw[opacity=0.4] (z)--(lc);
  \draw[dashed] (la)--(lb);
  \draw (lb)--(lc);
  \draw (lc)--(la);
  \draw (z)--(lad);
  \draw (z)--(lbd);
  
  \fill[fill=red,opacity=0.7] (lc)--(lcd)--(lact)--(barlac)--(lc);
  \fill[fill=red,opacity=0.7] (lc)--(lcd)--(lbct)--(barlbc)--(lc);
  \fill[fill=red,opacity=0.7] (lc)--(barlac)--(barlabc)--(barlbc)--(lc);
  \fill[fill=red,opacity=0.7] (barlac)--(lact)--(labcq)--(barlabc)--(barlac);
  \fill[fill=red,opacity=0.7] (barlbc)--(lbct)--(labcq)--(barlabc)--(barlbc);
  \draw[opacity=0.4] (lc)--(lcd);
  \draw (lc)--(barlbc);
  \draw (lc)--(barlac);
  \draw[opacity=0.4] (lact)--(barlac);
  \draw[opacity=0.4] (lbct)--(barlbc);
  \draw[opacity=0.4] (lact)--(labcq);
  \draw[opacity=0.4] (lbct)--(labcq);
  \draw (barlbc)--(barlabc);
  \draw (barlac)--(barlabc);
  \draw[opacity=0.4] (barlabc)--(labcq);
  \draw[opacity=0.4] (lcd)--(lbct);
  \draw[opacity=0.4] (lcd)--(lact);
  
\end{tikzpicture}
}
\caption{Fundamental domains $F_Y$ and $F_Y'$ inside ${\mathcal{A}_0}$ in type $A_3$.}\label{fig:interA3}
\end{figure}
\end{exemple}

\printbibliography

\end{document}